\documentclass[12pt, reqno]{amsart} 

\usepackage{fancybox,color, mathtools, array, amssymb}
\usepackage{amsthm}
\mathtoolsset{showonlyrefs,showmanualtags}
\usepackage[shortlabels]{enumitem}
\usepackage[margin=1.0in]{geometry}
\usepackage{mathrsfs}
\usepackage{amsmath}
\usepackage{lmodern}
\usepackage{wasysym}
\usepackage{mdframed}
\usepackage{tcolorbox}
\usepackage{setspace}
\usepackage{subcaption}
\tcbuselibrary{theorems}
\usepackage{subfiles}
\usepackage{tikz}
\usepackage[citestyle=alphabetic, bibstyle=alphabetic,maxnames=50,maxcitenames=50,maxalphanames=50]{biblatex}
\AtBeginBibliography{\small}
\usetikzlibrary{graphs}
\usetikzlibrary{trees}
\usepackage{hyperref}

\newtheorem{theorem}{Theorem}[section]
\newtheorem{lemma}{Lemma}[section]
\newtheorem{defn}{Definition}[section]
\newtheorem{propo}{Proposition}[section]
\newtheorem{coro}{Corollary}[section]

\theoremstyle{remark}

\newtheorem{rem}{Remark}[section]
\theoremstyle{definition}

\newcommand{\R}{\mathbb{R}}

\newcommand{\Z}{\mathbb{Z}}

\newcommand{\emp}{\emptyset}

\newcommand\pderiv[2]{\frac{\partial #1}{\partial #2}}
\newcommand\innerp[2]{\langle #1, #2 \rangle}
\newcommand\wtil[1]{\widetilde{#1}}

\newcommand{\ep}{\varepsilon}
\newcommand{\sub}{\subseteq}

\newcommand{\spa}{\text{span}}
\newcommand{\p}{\partial}
\newcommand{\grad}{\nabla}

\newcommand{\lap}{\Delta}

\title{Stability of critical equipartitions of graphs}
\author{Connor Menzel}
\date{}

\begin{document}

\begin{abstract}
    We consider partitions of a finite, simple, weighted graph that minimize a spectral energy functional, defined to be the maximum of the first eigenvalues on each component. These partitions are minimized with respect to a parameter that we view as a small perturbation of a fixed combinatorial partition. It has been shown that critical points of the energy functional in this framework correspond to non-degenerate eigenvectors of the graph Laplacian if and only if the partition is bipartite. In this work we generalize this result to partitions that are not necessarily bipartite by constructing a modified graph Laplacian called the \textit{partition Laplacian}. The main result states that critical points of the spectral energy functional correspond to nodal partitions of non-degenerate eigenvectors of the partition Laplacian. Furthermore, the stability of each critical point is determined entirely by the nodal deficiency of the associated eigenvector. We also consider more general signed Laplacians and prove that Courant-sharp Laplacian eigenvectors induce globally minimal partitions.
\end{abstract}

\maketitle

\section{Introduction} 

Given an eigenfunction of the Laplacian on a bounded domain or an eigenvector of the graph Laplacian on a finite discrete graph, it is natural to investigate the properties of its zero set, often called its \textit{nodal set}. Chladni was the first to systematically study nodal patterns on vibrating plates \cite{chladniEntdeckungenUeberTheorie1787}. Courant  then generalized the earlier work of Sturm \cite{sturmMemoireEquationsDifferentielles2009} to dimensions greater than one, proving that the $n$th eigenfunction of a second order self-adjoint differential operator divides the domain into no more than $n$ connected components on which the eigenfunction has a constant sign \cite{courantMethodsMathematicalPhysics2009}. These components are called nodal domains.

Counting the nodal domains of eigenfunctions has led to many exciting results. The nodal domain count is connected to quantum chaos \cite{blumNodalDomainsStatistics2002,bogomolnyPercolationModelNodal2002} and there is a growing body of work that conjectures that knowledge of the nodal domain counts could distinguish between isospectral domains \cite{gnutzmannCanOneCount2006,gnutzmannResolvingIsospectralDrums2005,brueningRemarksResolvingIsospectral2007} or graphs \cite{bandNodalDomainsGraphs2008}. There are also deep connections between the nodal domain count and minimality of a spectral energy functional \cite{helfferNodalDomainsSpectral2006a,berkolaikoStabilitySpectralPartitions2024a,berkolaikoStabilitySpectralPartitions2022,berkolaikoCriticalPartitionsNodal2012}. This connection in the discrete case is the main focus of this work. For a more complete survey of nodal properties of eigenfunctions, we refer the reader to \cite{jainNodalPortraitsQuantum2017}.

Here we are concerned with eigenvectors of the graph Laplacian on a finite graph.
Let $G = (V,E,w)$ be a finite, connected, weighted graph. We define the adjacency matrix $A_{ij} = w_{ij}$
where $w_{ij} = 0$ for all $(i,j) \notin E$ and $A_{ii} = 0$ for all $i \in V$. The degree matrix $D$ is the diagonal matrix given by $D_{ii} = \sum_{j \sim i}w_{ij}$,
where $j \sim i$ denotes that $(i,j) \in E$. We then define the graph Laplacian $L$ as 
\begin{equation}
    L = D - A.
\end{equation}
We identify vectors in $\R^{|V|}$ with functions $u:V \to \R$, and the graph Laplacian acts on $u$ by
\begin{equation}
    (Lu)_i = \sum_{j \sim i}w_{ij}(u_i - u_j),
\end{equation}
where we use the shorthand $u_i := u(i)$.
Alternatively, the Laplacian is generated by the quadratic form
\begin{equation}
    \innerp{u}{Lu} = \sum_{(i,j) \in E}w_{ij}(u_i-u_j)^2.
\end{equation}
    
The Laplacian is symmetric and positive semidefinite, and hence has nonnegative real eigenvalues, which we label
$0 = \lambda_1 \leq \lambda_2 \leq \cdots \leq \lambda_{|V|}$,
repeated according to multiplicity.

Given an eigenvector $\psi^{(n)}$ corresponding to eigenvalue $\lambda_n$, a \textit{strong nodal domain} is a maximally connected subgraph of $G$ on which $\psi^{(n)}$ has constant (nonzero) sign. A \textit{weak nodal domain} of $\psi^{(n)}$ is a maximally connected subgraph $H$ of $G$ on which $\psi^{(n)}_i \geq 0$ (or $\psi^{(n)}_i\leq 0$) for all $i$ in the vertex set of $H$. See Figure \ref{nodal_domain_example} for an example showing the difference between these two definitions.

In this paper, we will consider only strong nodal domains. We let $\nu(\psi^{(n)})$ denote the number of strong nodal domains of an eigenvector $\psi^{(n)}$. Davies, Gladwell, Leydold, and Stadler \cite{briandaviesDiscreteNodalDomain2001} proved a discrete version of Courant's nodal domain theorem: if $\psi^{(n)}$ is an eigenvalue corresponding to an eigenvalue $\lambda_n$ with multiplicity $r$, then 
\begin{equation}\label{discrete_nodal_domains}
    \nu(\psi^{(n)}) \leq n+r - 1.
\end{equation}
In the special case that $\lambda_n$ is simple, we recover the more standard nodal domain theorem that $\nu(\psi^{(n)}) \leq n$. We will only be considering simple eigenvalues, so we define the \textit{nodal deficiency} $\delta(\psi^{(n)})= n - \nu(\psi^{(n)})$.
By \eqref{discrete_nodal_domains}, the nodal deficiency is always nonnegative.

\begin{figure}
    \centering
    \begin{tikzpicture}[x=0.75pt,y=0.75pt,yscale=-1,xscale=1]

\draw    (194.69,73.44) -- (196.06,155.19) ;
\draw    (196.06,155.19) -- (196.74,236.94) ;
\draw    (112.94,154.51) -- (196.06,155.19) ;
\draw    (196.06,155.19) -- (279.17,155.88) ;
\draw  [fill={rgb, 255:red, 155; green, 155; blue, 155 }  ,fill opacity=1 ] (183.11,155.19) .. controls (183.11,148.05) and (188.91,142.25) .. (196.06,142.25) .. controls (203.2,142.25) and (209,148.05) .. (209,155.19) .. controls (209,162.34) and (203.2,168.14) .. (196.06,168.14) .. controls (188.91,168.14) and (183.11,162.34) .. (183.11,155.19) -- cycle ;
\draw  [fill={rgb, 255:red, 245; green, 166; blue, 35 }  ,fill opacity=1 ] (181.75,73.44) .. controls (181.75,66.3) and (187.55,60.5) .. (194.69,60.5) .. controls (201.84,60.5) and (207.64,66.3) .. (207.64,73.44) .. controls (207.64,80.59) and (201.84,86.39) .. (194.69,86.39) .. controls (187.55,86.39) and (181.75,80.59) .. (181.75,73.44) -- cycle ;
\draw  [fill={rgb, 255:red, 245; green, 166; blue, 35 }  ,fill opacity=1 ] (264.86,155.19) .. controls (264.86,148.05) and (270.66,142.25) .. (277.81,142.25) .. controls (284.95,142.25) and (290.75,148.05) .. (290.75,155.19) .. controls (290.75,162.34) and (284.95,168.14) .. (277.81,168.14) .. controls (270.66,168.14) and (264.86,162.34) .. (264.86,155.19) -- cycle ;
\draw  [fill={rgb, 255:red, 74; green, 144; blue, 226 }  ,fill opacity=1 ] (100,154.51) .. controls (100,147.36) and (105.8,141.57) .. (112.94,141.57) .. controls (120.09,141.57) and (125.89,147.36) .. (125.89,154.51) .. controls (125.89,161.66) and (120.09,167.46) .. (112.94,167.46) .. controls (105.8,167.46) and (100,161.66) .. (100,154.51) -- cycle ;
\draw  [fill={rgb, 255:red, 74; green, 144; blue, 226 }  ,fill opacity=1 ] (183.79,236.94) .. controls (183.79,229.8) and (189.59,224) .. (196.74,224) .. controls (203.89,224) and (209.68,229.8) .. (209.68,236.94) .. controls (209.68,244.09) and (203.89,249.89) .. (196.74,249.89) .. controls (189.59,249.89) and (183.79,244.09) .. (183.79,236.94) -- cycle ;

\draw (187,66.9) node [anchor=north west][inner sep=0.75pt]    {$+$};
\draw (270,148.4) node [anchor=north west][inner sep=0.75pt]    {$+$};
\draw (106,148) node [anchor=north west][inner sep=0.75pt]    {$-$};
\draw (190,230) node [anchor=north west][inner sep=0.75pt]    {$-$};
\draw (190.5,148) node [anchor=north west][inner sep=0.75pt]    {$0$};
\end{tikzpicture}
\caption{The nodal diagram for an eigenvector corresponding to the second eigenvalue on the star graph with 5 vertices. This eigenvector has 4 strong nodal domains but only two weak nodal domains.}\label{nodal_domain_example}
\end{figure}
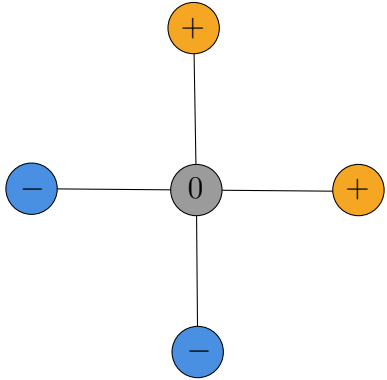

There is a deep connection between nodal domains of graph Laplacian eigenvectors and criticality of graph partitions with respect to a spectral energy functional, first established by Berkolaiko, Raz, and Smilansky \cite{berkolaikoStabilityNodalStructures2012}. To be more precise, we let a $\nu$-partition $P = \{G_k\}_{k=1}^{\nu}$ be a decomposition of $G$ into $\nu$ disjoint, connected, induced subgraphs. We consider small perturbations of $P$ depending smoothly on some parameter $\alpha$ in a sense that will be defined precisely in Section \ref{sec:param_part}. We then let $\lambda_1(G_k,\alpha)$ denote the first eigenvalue on $G_k$, which, in general, depends smoothly on $\alpha$. Define a spectral functional $\Lambda$ on parameter-dependent partitions as 
\[\Lambda(P,\alpha) = \max_{k}\lambda_1(G_k,\alpha).\]
In \cite{berkolaikoStabilityNodalStructures2012} the authors demonstrated that once restricted to a certain smooth submanifold of parameter-dependent partitions, the critical points of $\Lambda$ with respect to $\alpha$ correspond to nodal partitions of nondegenerate graph Laplacian eigenvectors if and only if the underlying partition $P$ is bipartite. Furthermore, the Morse index of any critical point is equal to the nodal deficiency of the corresponding eigenvector. Recall that the Morse index is the number of negative eigenvalues of the Hessian of $\Lambda$ at the critical point, and hence determines the stability of the critical point.

The above work relies crucially on the bipartite structure of the partition $P$, since every nodal partition of a graph Laplacian eigenvector must be bipartite; the bipartite structure is determined by the signs of the eigenvector on its nodal domains. We generalize this work to understand critical points of $\Lambda$ with respect to $\alpha$ when the underlying partition is not necessarily bipartite. Since we have no hope of relating non-bipartite partitions to nodal partitions of graph Laplacian eigenvectors, we consider eigenvectors of a signed graph Laplacian, which we call the \textit{partition Laplacian}.

With this framework, we can state a loose version of our main result, although the precise statement requires terminology that will not be discussed until Section \ref{sec:param_part}:

\begin{theorem}\label{thm:main_result}
    Given a partition $P$ of a graph $G$ into $\nu$ connected subgraphs $\{G_k\}_{k=1}^{\nu}$, consider local perturbations of $P$ parametrized by $\alpha$. Define 
    \[\Lambda(P,\alpha) = \max_{k=1,\dots,\nu}\lambda_1(G_k,\alpha).\]
    On a smooth submanifold of the parameter space, there is a one-to-one correspondence between critical points of the map $\alpha \mapsto \Lambda(P,\alpha)$ and eigenvectors of the partition Laplacian $L^{\p P}$ with nodal partition $P$. Furthermore, the Morse index of any critical point is equal to the nodal deficiency of the corresponding eigenvector.
\end{theorem}
This result provides a characterization of local minimal partitions on discrete graphs. It also provides an explicit method for computing the nodal deficiency of certain signed Laplacian eigenvectors. 
Graph partitioning has applications in clustering in data science \cite{ostingMinimalDirichletEnergy2014,ostingConsistencyDirichletPartitions2017}. 
This result can also be seen as a discrete analogue to \cite[Theorem 1.5]{berkolaikoStabilitySpectralPartitions2024a}.
As further motivation, we now summarize some results on minimal partitions of bounded domains in $\R^2$.

\subsection{Spectral minimal partitions}
This work on discrete graphs is very closely related to spectral minimal partitions of bounded domains in $\R^{2}$, and the definition of the partition Laplacian in Section \ref{part_lap_subsection} was inspired by the partition Laplacian in the continuum case. See Appendix \ref{app:connection_to_cts} for a more detailed explanation of the connection between the discrete and continuum partition Laplacians. Here we summarize some of the most crucial developments in the theory of spectral minimal partitions. 

Consider a bounded domain $\Omega \sub \R^2$ and let $\lambda_1 \leq \lambda_2 \leq \cdots$ denote the eigenvalues of the Dirichlet Laplacian on $\Omega$, repeated according to multiplicity. Given an eigenfunction $\psi_k$ with corresponding eigenvalue $\lambda_k$, we define a \textit{nodal domain} of $\psi_k$ to be a connected component of the set $\{x \in \Omega:\psi_k(x) \neq 0\}$ and we denote the number of nodal domains of $\psi_k$ by $\nu(\psi_k)$. 

Courant's nodal domain theorem \cite{courantMethodsMathematicalPhysics2009,chavelEigenvaluesRiemannianGeometry1984} states that 
\begin{equation}
    \nu(\psi_k) \leq k.
\end{equation}
That is, the number of nodal domains is bounded above by the index of the corresponding eigenvalue. We can again define the \textit{nodal deficiency} $\delta(\psi_k) = k - \nu(\psi_k)$,
which is necessarily nonnegative.
If $\nu(\psi_k) = k$ (or equivalently $\delta(\psi_k) = 0$) we say that $\psi_k$ is \textit{Courant-sharp}. It has been shown for $n\geq 2$ that any domain has at most finitely many Courant-sharp eigenfunctions \cite{pleijelRemarksCourantsNodal1956}.

A related topic, first systematically studied by Helffer, Hoffmann-Hostenhof, and Terracini \cite{helfferNodalDomainsSpectral2006a}, is that of \textit{spectral minimal partitions}. Given a partition $P$ of $\Omega$ into $\nu$ disjoint subdomains $\{\Omega_j\}_{j=1}^{\nu}$, let $\lambda_1(\Omega_j)$ denote the first eigenvalue of the Dirichlet Laplacian on $\Omega_j$. We define the energy functional
\[\Lambda(P) = \max_{j=1,\dots,\nu}\lambda_1(\Omega_j)\]
and define a \textit{spectral minimal $N$-partition} of $\Omega$ to be a partition that minimizes $\Lambda$ within the set of all $N$-partitions. Spectral minimal partitions arise naturally in the study of spatial segregation in the strong competition limit of reaction-diffusion systems \cite{contiClassOptimalPartition2003,contiAsymptoticEstimatesSpatial2005} and in the study of free boundary variational problems \cite{altVariationalProblemsTwo1984}. The main problem of interest is to characterize all minimal partitions for a given domain $\Omega$. In \cite{helfferNodalDomainsSpectral2006a}, the authors prove that a minimal partition corresponds to the nodal partition of a Courant-sharp eigenfunction if and only if the partition is bipartite. 

This has since been generalized greatly. If one restricts to a suitable manifold of partitions, the map $P \mapsto \Lambda(P)$ becomes smooth \cite{berkolaikoCriticalPartitionsNodal2012}. It follows that a bipartite partition $P$ is a critical point of $\Lambda$ if and only if it is the nodal partition of some Laplacian eigenfunction \cite{berkolaikoStabilitySpectralPartitions2022,berkolaikoStabilitySpectralPartitions2024a}. Furthermore, the Morse index of $\Lambda$ at this critical point is given precisely by the nodal deficiency of the corresponding eigenfunction.

The above analysis also works in the case of non-bipartite partitions, although more care is required. Given a partition $P$, one defines a modified Laplacian $-\Delta^{\p P}$ called the \textit{partition Laplacian}, which acts on functions satisfying anticontinuous boundary conditions on the partition boundary. A partition $P$ (not necessarily bipartite) is a critical point of $\Lambda$ if and only if it is the nodal partition of some eigenfunction of $-\Delta^{\p P}$ with nodal partition $P$. In the case that $P$ is bipartite, $-\Delta^{\p P}$ is unitarily equivalent to $-\Delta$, so the partition Laplacian allows us to study bipartite and non-bipartite partitions simultaneously. See \cite{berkolaikoStabilitySpectralPartitions2022}, \cite{berkolaikoStabilitySpectralPartitions2024a} and \cite{berkolaikoHomologySpectralMinimal2024} for more details about this approach. 

This work aims to characterize spectral minimal partitions on discrete graphs, but there has also been substantial progress in understanding partitions of metric graphs  \cite{kennedyTheorySpectralPartitions2020,bandConnectionNumberNodal2012} and, in fact, the work in \cite{bandConnectionNumberNodal2012} motivated some of the results on domains discussed above. 

\subsection{Outline} In Section \ref{sec:background} we begin with some background and notation from graph theory, introducing the concepts of signed graphs and signed Laplacians, which are the most essential tools in constructing $L^{\p P}$ and proving the main theorem. Section \ref{sec:param_part} then develops the notion of parameter-dependent partitions, which gives us a way to smoothly perturb a discrete graph partition. Section \ref{sec:main_result} provides a proof of the main theorem. In Section \ref{sec:lower_bound} we make generalizations that allow for a lower bound on the partition energy, without requiring a restriction to any smooth submanifold of partitions. As an immediate corollary, we show that Courant-sharp Laplacian eigenvectors induce globally minimal partitions.

\subsection{Acknowledgments} We are grateful for the advice of his thesis advisors Y. Canzani and J. Marzuola and would also like to thank G. Berkolaiko, who helpfully suggested that the results would benefit from the introduction of signed Laplacians. The author acknowledges the support of NSF CAREER Grant DMS-2045494, NSF FRG grant DMS-2152289, NSF Applied Math Grant DMS-2307384 and NSF RTG grant DMS-2135998.

\section{Definitions and Background}\label{sec:background}

A graph $G = (V,E)$ is a set of vertices $V$ along with a set of edges $E$ connecting them. If $i,j \in V$ are connected by some edge $e$, we write $e = (i,j) \in E$, and we commonly use the notation $i \sim j$ to mean that $(i,j) \in E$. A graph is simple if there is at most one edge between any two vertices and there are no edges from a vertex to itself. If a graph $G$ is not simple, we say that $G$ is a multigraph. Unless otherwise stated, all graphs will be finite, simple and connected. A graph $G$ is bipartite if there exists a function $\chi:V \to \{-1,+1\}$ such that $\chi_i\chi_j<0$ for all $(i,j) \in E$. A graph $G$ is a tree if there are no cycles.

We will consider weighted graphs $G = (V,E,w)$, where $w:E \to (0,\infty)$ is the edge weight function. We assume $w$ to be strictly positive, but when convenient we may instead consider $w:V \times V \to [0,\infty)$ with the added constraint that $w(i,j) \neq 0$ if and only if $(i,j) \in E$. We will also use the shorthand $w_{ij}:=w(i,j)$.

\subsection{Signed graphs}
One of the main tools we will use in proving the main theorem is the theory of signed graphs, which were first introduced and studied by Harary \cite{hararyNotionBalanceSigned1953} and have since led to advancements in various subfields of graph theory, such as matroid theory \cite{zaslavskySignedGraphs1982} and the study of Cheeger constants \cite{atayCheegerConstantsStructural2020}. Signed Laplacians have been used to construct Ramanujan graphs \cite{biluLiftsDiscrepancyNearly2006,marcusInterlacingFamiliesBipartite2015} and later contributed to the proof of the Kadison-Singer conjecture \cite{marcusInterlacingFamiliesII2015}.
For a more complete discussion of signed graphs, see \cite{zaslavskyCharacterizationsSignedGraphs1981}.

The necessity of signed graphs arises naturally when studying graph partitions, and we will find that the partition Laplacian $L^{\p P}$ is simply a signed Laplacian on a graph whose signature is negative on the partition boundary. The graph partition Laplacian can also be viewed as a discretization of the continuum partition Laplacian $-\lap^{\p P}$. For a detailed explanation we refer the reader to Appendix \ref{app:connection_to_cts}.

\begin{defn}\label{signed_graph}
    A signed graph $\wtil{G} = (G,\sigma)$ consists of a graph $G = (V,E,w)$ (called the underlying graph) and a function $\sigma:E \to \{-1,+1\}$, called the signature of $\wtil{G}$. 
\end{defn}

One could also define a signed graph as a weighted graph where we allow negative edge weights, but for our uses it is nicer to always assume the weights are positive. We will usually refer to a signed graph as $(G,\sigma)$, where we assume as before that $G$ is a finite, simple, connected and weighted.

\begin{defn}
    A function $\tau:V \to \{-1,+1\}$ is called a switching function. Given a signed graph $(G,\sigma)$, we use $\tau$ to switch the signature from $\sigma$ to $\sigma^{\tau}$ given by 
    \begin{equation}
        \sigma^{\tau}_{ij} = \tau_i\sigma_{ij}\tau_j.
    \end{equation}
    Two signatures $\sigma$ and $\sigma'$ on $G$ are said to be switching equivalent if there is some switching function $\tau$ such that $\sigma^{\tau} = \sigma'$.
\end{defn}
The effect of switching $\sigma$ by $\tau$ is simply reversing the sign of $\sigma$ on any edge across which $\tau$ changes sign. See Figure \ref{fig:switching_equivalence} for an illustration of switching equivalence.

We associate to each path $p = (e_1,\dots,e_n)$ in $(G,\sigma)$ a quantity $\sigma(p)$ called the sign of $p$, given by 
\begin{equation}
    \sigma(p) = \sigma_{e_1}\cdots \sigma_{e_n}.
\end{equation}

\begin{defn}
    A signed graph $(G,\sigma)$ is balanced if the sign of every cycle is positive.
\end{defn}

Zaslavsky showed that there is a connection between balanced sign graphs and graphs with strictly positive signature \cite{zaslavskySignedGraphs1982}:

\begin{lemma}[{\cite[Corollary 3.3]{zaslavskySignedGraphs1982}}]\label{zaslavsky} 
    A signed graph $(G,\sigma)$ is balanced if and only if $\sigma$ is switching equivalent to the all-positive signature.
\end{lemma}
\begin{figure}
    \begin{subfigure}[c]{0.3\textwidth}
        \centering
        \begin{tikzpicture}[x=0.75pt,y=0.75pt,yscale=-1,xscale=1]
        
        \draw    (90,107.95) -- (240.2,107.95) ;
        \draw [shift={(240.2,107.95)}, rotate = 0] [color={rgb, 255:red, 0; green, 0; blue, 0 }  ][fill={rgb, 255:red, 0; green, 0; blue, 0 }  ][line width=0.75]      (0, 0) circle [x radius= 3.35, y radius= 3.35]   ;
        \draw [shift={(90,107.95)}, rotate = 0] [color={rgb, 255:red, 0; green, 0; blue, 0 }  ][fill={rgb, 255:red, 0; green, 0; blue, 0 }  ][line width=0.75]      (0, 0) circle [x radius= 3.35, y radius= 3.35]   ;
        \draw    (164.17,188.7) -- (240.2,107.95) ;
        \draw    (90,107.95) -- (164.17,188.7) ;
        \draw [shift={(164.17,188.7)}, rotate = 47.44] [color={rgb, 255:red, 0; green, 0; blue, 0 }  ][fill={rgb, 255:red, 0; green, 0; blue, 0 }  ][line width=0.75]      (0, 0) circle [x radius= 3.35, y radius= 3.35]   ;
        
        \draw (155.65,87.34) node [anchor=north west][inner sep=0.75pt]    {$+$};
        \draw (107.22,143.87) node [anchor=north west][inner sep=0.75pt]    {$+$};
        \draw (208.78,142.99) node [anchor=north west][inner sep=0.75pt]    {$+$};

        \end{tikzpicture}

        \subcaption{}
    \end{subfigure}
    \begin{subfigure}[c]{0.3\textwidth}
        \centering
        \begin{tikzpicture}[x=0.75pt,y=0.75pt,yscale=-1,xscale=1]
        
        \draw    (90,107.95) -- (240.2,107.95) ;
        \draw [shift={(240.2,107.95)}, rotate = 0] [color={rgb, 255:red, 0; green, 0; blue, 0 }  ][fill={rgb, 255:red, 0; green, 0; blue, 0 }  ][line width=0.75]      (0, 0) circle [x radius= 3.35, y radius= 3.35]   ;
        \draw [shift={(90,107.95)}, rotate = 0] [color={rgb, 255:red, 0; green, 0; blue, 0 }  ][fill={rgb, 255:red, 0; green, 0; blue, 0 }  ][line width=0.75]      (0, 0) circle [x radius= 3.35, y radius= 3.35]   ;
        \draw    (164.17,188.7) -- (240.2,107.95) ;
        \draw    (90,107.95) -- (164.17,188.7) ;
        \draw [shift={(164.17,188.7)}, rotate = 47.44] [color={rgb, 255:red, 0; green, 0; blue, 0 }  ][fill={rgb, 255:red, 0; green, 0; blue, 0 }  ][line width=0.75]      (0, 0) circle [x radius= 3.35, y radius= 3.35]   ;
        
        \draw (155.65,87.34) node [anchor=north west][inner sep=0.75pt]    {$+$};
        \draw (107.22,143.87) node [anchor=north west][inner sep=0.75pt]    {$-$};
        \draw (208.78,142.99) node [anchor=north west][inner sep=0.75pt]    {$-$};

        \end{tikzpicture}
        \caption{}
    \end{subfigure}
    \begin{subfigure}[c]{0.3\textwidth}
        \centering
        \begin{tikzpicture}[x=0.75pt,y=0.75pt,yscale=-1,xscale=1]
        
        \draw    (90,107.95) -- (240.2,107.95) ;
        \draw [shift={(240.2,107.95)}, rotate = 0] [color={rgb, 255:red, 0; green, 0; blue, 0 }  ][fill={rgb, 255:red, 0; green, 0; blue, 0 }  ][line width=0.75]      (0, 0) circle [x radius= 3.35, y radius= 3.35]   ;
        \draw [shift={(90,107.95)}, rotate = 0] [color={rgb, 255:red, 0; green, 0; blue, 0 }  ][fill={rgb, 255:red, 0; green, 0; blue, 0 }  ][line width=0.75]      (0, 0) circle [x radius= 3.35, y radius= 3.35]   ;
        \draw    (164.17,188.7) -- (240.2,107.95) ;
        \draw    (90,107.95) -- (164.17,188.7) ;
        \draw [shift={(164.17,188.7)}, rotate = 47.44] [color={rgb, 255:red, 0; green, 0; blue, 0 }  ][fill={rgb, 255:red, 0; green, 0; blue, 0 }  ][line width=0.75]      (0, 0) circle [x radius= 3.35, y radius= 3.35]   ;
        
        \draw (155.65,87.34) node [anchor=north west][inner sep=0.75pt]    {$+$};
        \draw (107.22,143.87) node [anchor=north west][inner sep=0.75pt]    {$+$};
        \draw (208.78,142.99) node [anchor=north west][inner sep=0.75pt]    {$-$};

        \end{tikzpicture}
        \caption{}
    \end{subfigure}
    \caption{The 3-cycle graph $C_3$ with three distinct signatures. (A) and (B) are switching equivalent, but (C) is not switching equivalent to either because the sign of the cycle in (C) is negative.}
    \label{fig:switching_equivalence}
\end{figure}
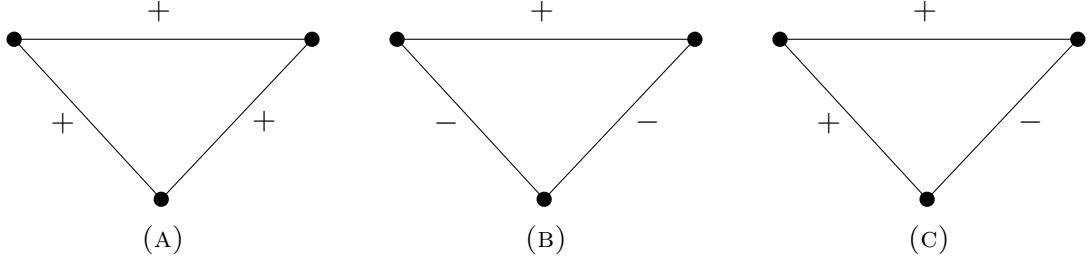
Harary \cite{hararyNotionBalanceSigned1953} also proved a characterization of balanced signed graphs that states that balanced graphs can be thought of as having some bipartite structure.

\begin{lemma}[{\cite[Theorem 3]{hararyNotionBalanceSigned1953}}]\label{harary}
    A signed graph $(G,\sigma)$ is balanced if and only if there exists a partition of $V$ into two disjoint subsets $V_1$ and $V_2$ such that any edge connecting $V_1$ and $V_2$ has negative signature and any other edge has positive signature.
\end{lemma}

Any signed graph $(G,\sigma)$ has an associated signed graph Laplacian matrix. If we define the signed adjacency matrix 
$(A^{\sigma})_{ij} = \sigma_{ij}w_{ij}$
and the diagonal degree matrix 
$D_{ii} = \sum_{j \sim i}w_{ij}$,
we define the signed graph Laplacian on $\wtil{G}$ by 
\begin{equation}
    L^{\sigma} = D - A^{\sigma}.
\end{equation}
We can also consider $L^{\sigma}$ as being generated by the quadratic form 
\begin{equation}\label{quad_form}
    \innerp{u}{L^{\sigma}u} = \sum_{(i,j) \in E}w_{ij}(u_i - \sigma_{ij}u_j)^2.
\end{equation}
It is immediate that $L^{\sigma}$ is symmetric and positive semidefinite. In the special case that $\sigma \equiv 1$, we obtain the graph Laplacian on $G$, denoted by $L$.
We label the eigenvalues of $L^{\sigma}$ as 
$\lambda_1(\sigma) \leq \cdots \leq \lambda_{|V|}(\sigma),$
repeated according to multiplicity, with corresponding eigenvectors $\{\psi^{(k)}\}_{k=1}^{|V|}$.

Given a switching function $\tau$, let $(D^{\tau})_{ii} = \tau_i$. Then a straightforward calculation shows that 
$L^{\sigma^{\tau}} = D^{\tau}L^{\sigma}D^{\tau}$,
so if $\sigma$ is switching equivalent to $\sigma'$ we find that $L^{\sigma'}$ and $L^{\sigma}$ are unitarily equivalent. In particular they have the same spectrum, and given an eigenvector $\psi$ of $L^{\sigma}$ it follows that $\tau \psi$ is an eigenvector of $L^{\sigma'}$ with the same eigenvalue. 

\begin{rem}
    The signed Laplacian is a special case of the more general magnetic Laplacian first introduced in \cite{liebFluxesLaplaciansKasteleyns1992}, where switching equivalence is replaced by gauge equivalence. In other words, any two switching equivalent signed Laplacians are in fact gauge equivalent. Magnetic Laplacians and, more generally, discrete magnetic Schr{\"o}dinger operators, are also very well studied; see \cite{alonSmoothCriticalPoints2025} for a detailed treatment. Berkolaiko and Colin de Verdière have used discrete magnetic Schr{\"o}dinger operators to prove bounds on the number of nodal domains \cite{berkolaikoLowerBoundNodal2008,verdiereMagneticInterpretationNodal2013}.
\end{rem}

\subsection{Nodal domains on signed graphs}

Let $(G,\sigma)$ be a signed graph and suppose that $u:V \to \R$ satisfies $u_i \neq 0$ for all $i \in V$. 

\begin{defn}\label{def:nodal_set}
    The nodal set $\mathcal{N} \sub E$ of $u$ is defined as 
    \[\mathcal{N} = \{(i,j) \in E:u_i\sigma_{ij}u_j < 0\}.\]
    Define $G' = (V,E\setminus \mathcal{N})$. Each connected component of $G'$ is called a nodal domain of $u$, and the number of nodal domains of $u$ is denoted by $\nu(u)$.
\end{defn}
In the case that $\sigma \equiv 1$, a nodal domain of $u$ is a connected induced subgraph on which $u$ has constant sign. This agrees with the usual notion of nodal domains on unsigned graphs. Note that the assumption that $u_i \neq 0$ allows us to only consider strong nodal domains as defined in the introduction. From now on, we will refer to strong nodal domains simply as nodal domains.

We are only concerned with the nodal domains of eigenvectors of $L^{\sigma}$ on $(G,\sigma)$. In order to make sense of Definition \ref{def:nodal_set} in this context, we restrict ourselves to only consider non-degenerate eigenvectors of $L^{\sigma}$. 
\begin{defn}
    An eigenvector $\psi$ of $L^{\sigma}$ is nondegenerate if $\psi_i \neq 0$ for all $i \in V$ and $\psi$ corresponds to a simple eigenvalue.
\end{defn}
All eigenvectors considered for the rest of this paper are assumed to be non-degenerate. This property is generic with respect to perturbation of the edge weights.

It is helpful to note that the nodal domains are invariant under switching transformations.

\begin{lemma}[{\cite[Theorem 3.10]{geNodalDomainTheorems2023}}]
    Let $\psi$ be a non-degenerate eigenvector of $L^{\sigma}$ on $(G,\sigma)$ and suppose that $\tau$ is a switching function. Then an induced subgraph of $G$ is a nodal domain of $\psi$ on $(G,\sigma)$ if and only if it is a nodal domain of $\tau \psi$ on $(G,\sigma^{\tau})$.
\end{lemma}

Mohammadian also proved a version of Courant's nodal domain theorem for eigenvectors of symmetric matrices \cite{mohammadianGraphsTheirReal2016}. We only use the result as applied to signed Laplacians. The nodal domain theorem is also proven by Ge and Liu \cite{geNodalDomainTheorems2023}, using the framework of signed graphs as in this paper.

\begin{theorem}[{\cite[Corollary 3]{mohammadianGraphsTheirReal2016}}]\label{mohammadian}
    If $\psi$ is a non-degenerate eigenvector of $L^{\sigma}$ with corresponding eigenvalue $\lambda_n(\sigma)$, then 
    \[\nu(\psi) \leq n.\]
\end{theorem}

Given a non-degenerate eigenvector $\psi$ of $L^{\sigma}$ with eigenvalue $\lambda_n(\sigma)$ we define the nodal deficiency $\delta(\psi)$ as 
$\delta(\psi) = n - \nu(\psi).$
Theorem \ref{mohammadian} immediately implies that $\delta(\psi) \geq 0$.

\subsection{Graph partitions and the partition Laplacian}\label{part_lap_subsection}

Given a graph $G$, we are concerned with partitions of $G$, which we take to be decompositions of $G$ into disjoint connected components. A \textit{subgraph} $G'$ of $G$ is itself a graph with vertex set $V' \sub V$ and edge set $E' \sub E$. A \textit{spanning subgraph} of $G$ is a subgraph of $G$ with the same vertex set.

\begin{defn}
    A $\nu$-partition $P$ of $G = (V,E,w)$ is a spanning subgraph of $G$ with $\nu$ connected components $\{G_k = (V_k,E_k,w)\}_{k=1}^{\nu}$ such that for any $k$, if $i,j \in V_k$ and $(i,j) \in E$, then $(i,j) \in E_k$.
\end{defn}

In other words, the only edges we remove from $G$ to form $P$ are the edges between the desired connected components of $P$. Given a $\nu$-partition $P$ of $G$, we let $\p P$ denote the partition boundary, the set of all edges removed from $G$ to form $P$.

To each partition $P$ we associate a multigraph $P^G$, which has $\nu$ vertices corresponding to the connected components of $P$, and $(i,j) \in E(P^G)$ if and only if there are distinct connected components $G_k$ and $G_{\ell}$ of $P$ such that $i \in G_k$, $j \in G_{\ell}$, and $(i,j) \in E$. Note that the edges of $P^G$ are in one-to-one correspondence with the partition boundary $\p P$.
We say that a partition $P$ is bipartite if the multigraph $P^G$ is bipartite, and we say that $P$ is a tree partition if the associated $P^G$ is a tree. See Figure \ref{fig:partition_graph} for an example of a partition $P$ and its associated multigraph $P^G$.

Every eigenvector of a signed Laplacian $L^{\sigma}$ induces a partition called the \textit{nodal partition}, obtained by setting $\p P = \mathcal{N}$, where $\mathcal{N}$ is the nodal set of $\psi$ as in Definition \ref{def:nodal_set}. We will extensively refer to the nodal set in the rest of this work.

\begin{figure}
    \begin{subfigure}[c]{0.3\textwidth}
    \centering
    \begin{tikzpicture}[x=0.75pt,y=0.75pt,yscale=-1,xscale=1]
    
    \draw [color={rgb, 255:red, 245; green, 166; blue, 35 }  ,draw opacity=1 ][line width=1.5]    (313.43,146.35) -- (254.77,205.58) ;
    \draw [color={rgb, 255:red, 245; green, 166; blue, 35 }  ,draw opacity=1 ][line width=1.5]    (213.61,120.63) -- (228.43,177.35) ;
    \draw [color={rgb, 255:red, 245; green, 166; blue, 35 }  ,draw opacity=1 ][line width=1.5]    (273.9,145.67) -- (232,178.88) ;
    \draw [color={rgb, 255:red, 245; green, 166; blue, 35 }  ,draw opacity=1 ][line width=1.5]    (274.92,140.56) -- (217.7,116.54) ;
    \draw  [line width=1.5]  (192.15,93.55) .. controls (192.15,91.29) and (193.98,89.46) .. (196.24,89.46) .. controls (198.49,89.46) and (200.32,91.29) .. (200.32,93.55) .. controls (200.32,95.81) and (198.49,97.64) .. (196.24,97.64) .. controls (193.98,97.64) and (192.15,95.81) .. (192.15,93.55) -- cycle ;
    \draw  [line width=1.5]  (227.92,91) .. controls (227.92,88.74) and (229.75,86.91) .. (232,86.91) .. controls (234.26,86.91) and (236.09,88.74) .. (236.09,91) .. controls (236.09,93.25) and (234.26,95.08) .. (232,95.08) .. controls (229.75,95.08) and (227.92,93.25) .. (227.92,91) -- cycle ;
    \draw  [line width=1.5]  (209.52,116.54) .. controls (209.52,114.29) and (211.35,112.46) .. (213.61,112.46) .. controls (215.87,112.46) and (217.7,114.29) .. (217.7,116.54) .. controls (217.7,118.8) and (215.87,120.63) .. (213.61,120.63) .. controls (211.35,120.63) and (209.52,118.8) .. (209.52,116.54) -- cycle ;
    \draw  [line width=1.5]  (224.34,181.44) .. controls (224.34,179.18) and (226.17,177.35) .. (228.43,177.35) .. controls (230.68,177.35) and (232.51,179.18) .. (232.51,181.44) .. controls (232.51,183.69) and (230.68,185.52) .. (228.43,185.52) .. controls (226.17,185.52) and (224.34,183.69) .. (224.34,181.44) -- cycle ;
    \draw  [line width=1.5]  (203.9,211.58) .. controls (203.9,209.33) and (205.73,207.5) .. (207.99,207.5) .. controls (210.24,207.5) and (212.08,209.33) .. (212.08,211.58) .. controls (212.08,213.84) and (210.24,215.67) .. (207.99,215.67) .. controls (205.73,215.67) and (203.9,213.84) .. (203.9,211.58) -- cycle ;
    \draw  [line width=1.5]  (247.33,208.52) .. controls (247.33,206.26) and (249.16,204.43) .. (251.42,204.43) .. controls (253.68,204.43) and (255.51,206.26) .. (255.51,208.52) .. controls (255.51,210.78) and (253.68,212.61) .. (251.42,212.61) .. controls (249.16,212.61) and (247.33,210.78) .. (247.33,208.52) -- cycle ;
    \draw  [line width=1.5]  (273.39,143.11) .. controls (273.39,140.86) and (275.22,139.03) .. (277.48,139.03) .. controls (279.74,139.03) and (281.57,140.86) .. (281.57,143.11) .. controls (281.57,145.37) and (279.74,147.2) .. (277.48,147.2) .. controls (275.22,147.2) and (273.39,145.37) .. (273.39,143.11) -- cycle ;
    \draw  [line width=1.5]  (294.34,111.95) .. controls (294.34,109.69) and (296.17,107.86) .. (298.43,107.86) .. controls (300.69,107.86) and (302.52,109.69) .. (302.52,111.95) .. controls (302.52,114.2) and (300.69,116.03) .. (298.43,116.03) .. controls (296.17,116.03) and (294.34,114.2) .. (294.34,111.95) -- cycle ;
    \draw  [line width=1.5]  (312.22,143.11) .. controls (312.22,140.86) and (314.05,139.03) .. (316.31,139.03) .. controls (318.57,139.03) and (320.4,140.86) .. (320.4,143.11) .. controls (320.4,145.37) and (318.57,147.2) .. (316.31,147.2) .. controls (314.05,147.2) and (312.22,145.37) .. (312.22,143.11) -- cycle ;
    \draw [line width=1.5]    (231.49,183.99) -- (248.86,204.94) ;
    \draw [line width=1.5]    (225.87,185.01) -- (207.99,207.5) ;
    \draw [line width=1.5]    (247.33,208.52) -- (212.08,211.58) ;
    \draw [line width=1.5]    (199.3,96.11) -- (211.05,113.99) ;
    \draw [line width=1.5]    (200.32,93.55) -- (227.92,91) ;
    \draw [line width=1.5]    (229.45,93.55) -- (216.67,113.99) ;
    \draw [line width=1.5]    (295.87,114.5) -- (277.48,139.03) ;
    \draw [line width=1.5]    (312.22,143.11) -- (281.06,142.09) ;
    \draw [line width=1.5]    (300.98,115.01) -- (316.31,139.03) ;
    \draw  [color={rgb, 255:red, 74; green, 144; blue, 226 }  ,draw opacity=1 ][dash pattern={on 1.69pt off 2.76pt}][line width=1.5]  (258.32,130.9) .. controls (258.32,115.32) and (275.6,102.69) .. (296.91,102.69) .. controls (318.22,102.69) and (335.5,115.32) .. (335.5,130.9) .. controls (335.5,146.48) and (318.22,159.12) .. (296.91,159.12) .. controls (275.6,159.12) and (258.32,146.48) .. (258.32,130.9) -- cycle ;
    \draw  [color={rgb, 255:red, 74; green, 144; blue, 226 }  ,draw opacity=1 ][dash pattern={on 1.69pt off 2.76pt}][line width=1.5]  (175.85,101.28) .. controls (175.85,85.7) and (193.13,73.07) .. (214.44,73.07) .. controls (235.75,73.07) and (253.03,85.7) .. (253.03,101.28) .. controls (253.03,116.87) and (235.75,129.5) .. (214.44,129.5) .. controls (193.13,129.5) and (175.85,116.87) .. (175.85,101.28) -- cycle ;
    \draw  [color={rgb, 255:red, 74; green, 144; blue, 226 }  ,draw opacity=1 ][dash pattern={on 1.69pt off 2.76pt}][line width=1.5]  (188.63,198.85) .. controls (188.63,183.27) and (205.91,170.64) .. (227.22,170.64) .. controls (248.53,170.64) and (265.81,183.27) .. (265.81,198.85) .. controls (265.81,214.43) and (248.53,227.07) .. (227.22,227.07) .. controls (205.91,227.07) and (188.63,214.43) .. (188.63,198.85) -- cycle ;
    
    \draw (186.64,145.22) node [anchor=north west][inner sep=0.75pt]  [color={rgb, 255:red, 245; green, 166; blue, 35 }  ,opacity=1 ]  {$\partial P$};
    \end{tikzpicture}
    \end{subfigure}
    \begin{subfigure}[c]{0.3\textwidth}
        \centering
        \begin{tikzpicture}[x=0.75pt,y=0.75pt,yscale=-1,xscale=1]
        
        \draw  [line width=1.5]  (212.15,113.55) .. controls (212.15,111.29) and (213.98,109.46) .. (216.24,109.46) .. controls (218.49,109.46) and (220.32,111.29) .. (220.32,113.55) .. controls (220.32,115.81) and (218.49,117.64) .. (216.24,117.64) .. controls (213.98,117.64) and (212.15,115.81) .. (212.15,113.55) -- cycle ;
        \draw  [line width=1.5]  (247.92,111) .. controls (247.92,108.74) and (249.75,106.91) .. (252,106.91) .. controls (254.26,106.91) and (256.09,108.74) .. (256.09,111) .. controls (256.09,113.25) and (254.26,115.08) .. (252,115.08) .. controls (249.75,115.08) and (247.92,113.25) .. (247.92,111) -- cycle ;
        \draw  [line width=1.5]  (229.52,136.54) .. controls (229.52,134.29) and (231.35,132.46) .. (233.61,132.46) .. controls (235.87,132.46) and (237.7,134.29) .. (237.7,136.54) .. controls (237.7,138.8) and (235.87,140.63) .. (233.61,140.63) .. controls (231.35,140.63) and (229.52,138.8) .. (229.52,136.54) -- cycle ;
        \draw  [line width=1.5]  (244.34,201.44) .. controls (244.34,199.18) and (246.17,197.35) .. (248.43,197.35) .. controls (250.68,197.35) and (252.51,199.18) .. (252.51,201.44) .. controls (252.51,203.69) and (250.68,205.52) .. (248.43,205.52) .. controls (246.17,205.52) and (244.34,203.69) .. (244.34,201.44) -- cycle ;
        \draw  [line width=1.5]  (223.9,231.58) .. controls (223.9,229.33) and (225.73,227.5) .. (227.99,227.5) .. controls (230.24,227.5) and (232.08,229.33) .. (232.08,231.58) .. controls (232.08,233.84) and (230.24,235.67) .. (227.99,235.67) .. controls (225.73,235.67) and (223.9,233.84) .. (223.9,231.58) -- cycle ;
        \draw  [line width=1.5]  (267.33,228.52) .. controls (267.33,226.26) and (269.16,224.43) .. (271.42,224.43) .. controls (273.68,224.43) and (275.51,226.26) .. (275.51,228.52) .. controls (275.51,230.78) and (273.68,232.61) .. (271.42,232.61) .. controls (269.16,232.61) and (267.33,230.78) .. (267.33,228.52) -- cycle ;
        \draw  [line width=1.5]  (293.39,163.11) .. controls (293.39,160.86) and (295.22,159.03) .. (297.48,159.03) .. controls (299.74,159.03) and (301.57,160.86) .. (301.57,163.11) .. controls (301.57,165.37) and (299.74,167.2) .. (297.48,167.2) .. controls (295.22,167.2) and (293.39,165.37) .. (293.39,163.11) -- cycle ;
        \draw  [line width=1.5]  (314.34,131.95) .. controls (314.34,129.69) and (316.17,127.86) .. (318.43,127.86) .. controls (320.69,127.86) and (322.52,129.69) .. (322.52,131.95) .. controls (322.52,134.2) and (320.69,136.03) .. (318.43,136.03) .. controls (316.17,136.03) and (314.34,134.2) .. (314.34,131.95) -- cycle ;
        \draw  [line width=1.5]  (332.22,163.11) .. controls (332.22,160.86) and (334.05,159.03) .. (336.31,159.03) .. controls (338.57,159.03) and (340.4,160.86) .. (340.4,163.11) .. controls (340.4,165.37) and (338.57,167.2) .. (336.31,167.2) .. controls (334.05,167.2) and (332.22,165.37) .. (332.22,163.11) -- cycle ;
        \draw [line width=1.5]    (251.49,203.99) -- (268.86,224.94) ;
        \draw [line width=1.5]    (245.87,205.01) -- (227.99,227.5) ;
        \draw [line width=1.5]    (267.33,228.52) -- (232.08,231.58) ;
        \draw [line width=1.5]    (219.3,116.11) -- (231.05,133.99) ;
        \draw [line width=1.5]    (220.32,113.55) -- (247.92,111) ;
        \draw [line width=1.5]    (249.45,113.55) -- (236.67,133.99) ;
        \draw [line width=1.5]    (315.87,134.5) -- (297.48,159.03) ;
        \draw [line width=1.5]    (332.22,163.11) -- (301.06,162.09) ;
        \draw [line width=1.5]    (320.98,135.01) -- (336.31,159.03) ;
        \draw  [color={rgb, 255:red, 74; green, 144; blue, 226 }  ,draw opacity=1 ][dash pattern={on 1.69pt off 2.76pt}][line width=1.5]  (278.32,150.9) .. controls (278.32,135.32) and (295.6,122.69) .. (316.91,122.69) .. controls (338.22,122.69) and (355.5,135.32) .. (355.5,150.9) .. controls (355.5,166.48) and (338.22,179.12) .. (316.91,179.12) .. controls (295.6,179.12) and (278.32,166.48) .. (278.32,150.9) -- cycle ;
        \draw  [color={rgb, 255:red, 74; green, 144; blue, 226 }  ,draw opacity=1 ][dash pattern={on 1.69pt off 2.76pt}][line width=1.5]  (195.85,121.28) .. controls (195.85,105.7) and (213.13,93.07) .. (234.44,93.07) .. controls (255.75,93.07) and (273.03,105.7) .. (273.03,121.28) .. controls (273.03,136.87) and (255.75,149.5) .. (234.44,149.5) .. controls (213.13,149.5) and (195.85,136.87) .. (195.85,121.28) -- cycle ;
        \draw  [color={rgb, 255:red, 74; green, 144; blue, 226 }  ,draw opacity=1 ][dash pattern={on 1.69pt off 2.76pt}][line width=1.5]  (208.63,218.85) .. controls (208.63,203.27) and (225.91,190.64) .. (247.22,190.64) .. controls (268.53,190.64) and (285.81,203.27) .. (285.81,218.85) .. controls (285.81,234.43) and (268.53,247.07) .. (247.22,247.07) .. controls (225.91,247.07) and (208.63,234.43) .. (208.63,218.85) -- cycle ;
        
        \draw (330,135.02) node [anchor=north west][inner sep=0.75pt]  [color={rgb, 255:red, 0; green, 0; blue, 0 }  ,opacity=1 ]  {$G_{1}$};
        \draw (222,95) node [anchor=north west][inner sep=0.75pt]  [color={rgb, 255:red, 0; green, 0; blue, 0 }  ,opacity=1 ]  {$G_{2}$};
        \draw (260.04,200.82) node [anchor=north west][inner sep=0.75pt]  [color={rgb, 255:red, 0; green, 0; blue, 0 }  ,opacity=1 ]  {$G_{3}$};
        \end{tikzpicture}
    \end{subfigure}
    \begin{subfigure}[c]{0.3\textwidth}
    \centering
    \begin{tikzpicture}[x=0.75pt,y=0.75pt,yscale=-1,xscale=1]
    
    \draw [color={rgb, 255:red, 245; green, 166; blue, 35 }  ,draw opacity=1 ][line width=1.5]    (296.18,139.17) -- (189.02,100.44) ;
    \draw [color={rgb, 255:red, 245; green, 166; blue, 35 }  ,draw opacity=1 ][line width=1.5]    (209.97,206.76) -- (189.02,100.44) ;
    \draw [color={rgb, 255:red, 245; green, 166; blue, 35 }  ,draw opacity=1 ][line width=1.5]    (209.97,206.76) .. controls (206.27,166.44) and (258.31,130.82) .. (296.18,139.17) ;
    \draw [color={rgb, 255:red, 245; green, 166; blue, 35 }  ,draw opacity=1 ][line width=1.5]    (296.18,139.17) .. controls (300.53,155.05) and (270.72,210.49) .. (209.97,206.76) ;
    \draw  [color={rgb, 255:red, 74; green, 144; blue, 226 }  ,draw opacity=1 ][fill={rgb, 255:red, 74; green, 144; blue, 226 }  ,fill opacity=1 ] (175.4,100.44) .. controls (175.4,93.35) and (181.5,87.6) .. (189.02,87.6) .. controls (196.55,87.6) and (202.65,93.35) .. (202.65,100.44) .. controls (202.65,107.53) and (196.55,113.28) .. (189.02,113.28) .. controls (181.5,113.28) and (175.4,107.53) .. (175.4,100.44) -- cycle ;
    \draw  [color={rgb, 255:red, 74; green, 144; blue, 226 }  ,draw opacity=1 ][fill={rgb, 255:red, 74; green, 144; blue, 226 }  ,fill opacity=1 ] (196.35,206.76) .. controls (196.35,199.67) and (202.45,193.92) .. (209.97,193.92) .. controls (217.49,193.92) and (223.59,199.67) .. (223.59,206.76) .. controls (223.59,213.85) and (217.49,219.6) .. (209.97,219.6) .. controls (202.45,219.6) and (196.35,213.85) .. (196.35,206.76) -- cycle ;
    \draw  [color={rgb, 255:red, 74; green, 144; blue, 226 }  ,draw opacity=1 ][fill={rgb, 255:red, 74; green, 144; blue, 226 }  ,fill opacity=1 ] (282.55,139.17) .. controls (282.55,132.08) and (288.65,126.33) .. (296.18,126.33) .. controls (303.7,126.33) and (309.8,132.08) .. (309.8,139.17) .. controls (309.8,146.26) and (303.7,152.01) .. (296.18,152.01) .. controls (288.65,152.01) and (282.55,146.26) .. (282.55,139.17) -- cycle ;
    
    \draw (290.91,131.26) node [anchor=north west][inner sep=0.75pt]    {$1$};
    \draw (183.35,92.15) node [anchor=north west][inner sep=0.75pt]    {$2$};
    \draw (204.7,198.47) node [anchor=north west][inner sep=0.75pt]    {$3$};
    \end{tikzpicture}
    \end{subfigure}
    \caption{A simple graph $G$ with desired connected components $G_k$ of 3-partition $P$ circled, the resultant 3-partition $P$, and the associated multigraph $P^G$.}
    \label{fig:partition_graph}
\end{figure}
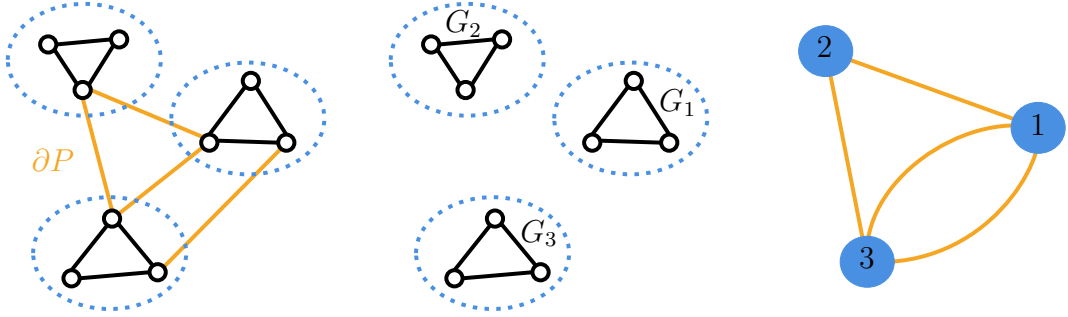

Given a partition $P$ we define a signature $\sigma^{\p P}$ on $G$ by 
\[\sigma^{\p P}_{ij} = \begin{cases}
    -1, & (i,j) \in \p P \\
    1, & (i,j) \in E \setminus \p P.
\end{cases}\]

\begin{defn}\label{def:part_lap}
    The signed Laplacian $L^{\sigma^{\p P}}$ on $(G,\sigma^{\p P})$ is called the partition Laplacian, and for ease of notation we denote it as $L^{\p P}$.
\end{defn}

From \eqref{quad_form} we have that 
\begin{equation}
    \innerp{u}{L^{\p P}u} = \sum_{(i,j) \notin \p P}w_{ij}(u_i - u_j)^2 + \sum_{(i,j) \in \p P}w_{ij}(u_i + u_j)^2.
\end{equation}
In the case that $P$ is bipartite, a straightforward application of Lemma \ref{harary} followed by Lemma \ref{zaslavsky} shows that $\sigma^{\p P}$ is switching equivalent to $\sigma \equiv 1$, so $L^{\p P}$ is unitarily equivalent to the Laplacian $L$ on $G$.

Suppose that $\psi$ is an eigenvector of $L$ and let $P$ be its associated nodal partition. Note that $P$ is necessarily bipartite since $\psi$ must have a constant sign on each nodal domain. Define a switching function $\tau:V \to \{-1,+1\}$ by $\tau_i = \text{sgn}(\psi_i)$. By construction, $\tau$ changes sign exactly across $\p P$, so $\tau$ switches $\sigma \equiv 1$ to $\sigma^{\tau} = \sigma^{\p P}$. We thus find that $\tau \psi$ is an eigenvector of $L^{\p P}$ that is positive everywhere.

Similarly, if we are given a non-bipartite partition $P$, note that any $(i,j) \in \p P$ is in the nodal set of an eigenvector $\psi$ of $L^{\p P}$ if and only if $\psi_i\psi_j>0$. Hence, if $\psi$ is postive everywhere, it must have nodal partition $P$. This is the crucial fact that allows us to relate non-bipartite partitions to nodal partitions of partition Laplacian eigenvectors.

\section{Parameter-Dependent Partitions}
\label{sec:param_part}

In order to state and prove Theorem \ref{thm:main_result}, we need to develop some techniques that allow us consider smooth perturbations of a graph partition. These strategies were developed in \cite{berkolaikoStabilityNodalStructures2012} but need to be adapted slightly to the case of the partition Laplacian. For sake of completeness, we give the arguments in full here.

Let $P$ be a $\nu$-partition of $G$ and let $L^{\p P}$ be the partition Laplacian as defined in Section \ref{part_lap_subsection}. For the rest of the paper, this partition $P$ is fixed and we only consider small perturbations of $P$ in a sense that we will define in this section.

\subsection{Removal of a single edge}
Suppose that $\psi$ is a non-degenerate eigenvector of $L^{\p P}$ with eigenvalue $\lambda_n$, and let $(i,j) \in \p P$. We will remove $(i,j)$ from our graph, but ask how we can perturb the partition Laplacian appropriately so that $\psi$ remains an eigenvector.
Define $q_{ij} = \psi_j/\psi_i$ and let $B_{ij}$ be the $|V|\times |V|$ matrix consisting of zeros everywhere except
\[(B_{ij})_{ii} = q_{ij} \hspace{1cm} (B_{ij})_{ij} = -1\]
\[(B_{ij})_{ji} =-1 \hspace{1cm} (B_{ij})_{jj} = q_{ji}.\]
A straightforward calculation shows that $\psi \in \ker(B_{ij})$, so $\psi$ is still an eigenvector of $L^{\p P} + w_{ij}B_{ij}$ with the same eigenvalue. Also note that perturbing $L^{\p P}$ by $w_{ij}B_{ij}$ effectively removes the edge $(i,j)$ since $(L^{\p P} + w_{ij}B_{ij})_{ij} = w_{ij} - w_{ij} = 0.$

We can generalize this by allowing $B_{ij}$ to depend smoothly on some parameter $\alpha\in \R\setminus \{0\}$ in place of $q_{ij}$. Let
\begin{equation}\label{parameter_perturbation}
    B_{ij}(\alpha) = \begin{pmatrix}
    \alpha & -1 \\
    -1 & 1/\alpha
    \end{pmatrix}
\end{equation}
and define
\begin{equation}\label{perturbed_part_lap}
    L^{\p P}(G',\alpha) = L^{\p P}(G) + w_{ij}B_{ij}(\alpha),
\end{equation}
where $L^{\p P}(G)$ is the partition Laplacian on the base graph $G$ and $G'$ is the resultant graph when $(i,j)$ is removed from $G$. To be explicit, when we write $B_{ij}(\alpha)$ we will always take $j>i$, so that the potential added on vertex $j$ will always be $1/\alpha$.

When $\alpha = \psi_j/\psi_i$ see that $\psi \in \ker(B_{ij}(\alpha))$ so $\psi$ is an eigenvector of $L^{\p P}$. Note that 
$\innerp{u}{B_{ij}(\alpha)u} = \alpha(u_i - u_j/\alpha)^2$,
so $B_{ij}(\alpha)$ is positive semidefinite when $\alpha>0$ and negative semidefinite when $\alpha < 0$.  In either case, $B_{ij}(\alpha)$ is a rank-1 perturbation. We then have at our disposal the Weyl interlacing theorem \cite[Corollary 4.3.9]{hornMatrixAnalysis2012} to bound the eigenvalues of $L^{\p P}(G',\alpha)$ between consecutive eigenvalues of $L^{\p P}(G)$.

\begin{theorem}[Weyl Interlacing]\label{weyl_interlacing}
    Let $\lambda_{m}(G',\alpha)$ denote the $m^{th}$ eigenvalue of $L^{\p P}(G',\alpha)$. If $\alpha > 0$, we have 
    \begin{equation}
        \lambda_m(G) \leq \lambda_{m}(G',\alpha) \leq \lambda_{m+1}(G).
    \end{equation}
    If $\alpha < 0$,
    \begin{equation}
        \lambda_{m-1}(G) \leq \lambda_{m}(G',\alpha) \leq \lambda_m(G).
    \end{equation}
\end{theorem}

We now consider the eigenvalue curves of $L^{\p P}(G',\alpha)$ as functions of $\alpha$, i.e. the maps $\alpha \mapsto \lambda_m(G',\alpha)$. We also write $\psi(\alpha)$ to denote an eigenvector associated to $\lambda_m(G',\alpha)$. Standard finite-dimensional perturbation theory shows that each $\lambda_m(G',\alpha)$ is smooth except possibly at some isolated algebraic singularities that arise when the branches cross \cite[Chapter 2.1]{katoPerturbationTheoryLinear1995}. In other words, if $\lambda_m(G',\alpha_0)$ is a simple eigenvalue of $L^{\p P}(G',\alpha_0)$, there is some neighborhood of $\alpha_0$ on which $\lambda_m(G',\alpha)$ is smooth. 

The following theorem shows that there is a direct correspondence between critical points of the eigenvalues of $L^{\p P}(G',\alpha)$ viewed as functions of $\alpha$, and non-degenerate eigenvectors of $L^{\p P}$.

\begin{theorem}\label{thm:crit_points_and_efuns}
    Suppose that $\lambda_m(G',\wtil{\alpha})$ is a simple eigenvalue of $L^{\p P}(G',\wtil{\alpha})$ with corresponding eigenvector $\psi(\wtil{\alpha})$. Then in a neighborhood of $\wtil{\alpha}$ not containing zero, any critical point $\alpha^c$ of the map $\alpha \mapsto \lambda_m(G',\alpha)$ satisfies either
    \begin{equation}\label{pos_crit_point}
        \alpha^c = \frac{\psi_j(\alpha^c)}{\psi_i(\alpha^c)}
    \end{equation}
    or 
    \begin{equation}\label{neg_crit_point}
        \alpha^c = -\frac{\psi_j(\alpha^c)}{\psi_i(\alpha^c)}.
    \end{equation}
    If $\alpha^c$ satisfies \eqref{pos_crit_point}, $\psi(\alpha^c)$ is an eigenvector of $L^{\p P}(G)$ with eigenvalue $\lambda_m(G',\alpha^c)$ and corresponding index $n$ satisfying $|m-n|\leq 1$.

    Conversely, if $\psi$ is a non-degenerate eigenvector of $L^{\p P}(G)$ with eigenvalue $\lambda_n(G)$, then $\psi$ is an eigenvector of $L^{\p P}(G',\alpha_0)$ with eigenvalue $\lambda_n(G)$, where 
    \begin{equation}
        \alpha_0 = \frac{\psi_j}{\psi_i},
    \end{equation}
    with corresponding index $m$ satisfying $|n-m|\leq 1$. If $\lambda_m(G',\alpha_0)$ is a simple eigenvalue of $L^{\p P}(G',\alpha_0)$, then $\alpha_0$ is a critical point of $\alpha \mapsto \lambda_m(G',\alpha)$.
\end{theorem}

\begin{proof}
    From standard finite-dimensional perturbation theory, if we assume $\psi(\alpha)$ is normalized we have 
    \begin{equation}
        \lambda_m'(G',\alpha) = \innerp{\psi(\alpha)}{B_{ij}'(\alpha)\psi(\alpha)} = |\psi_i(\alpha)|^2 - \frac{1}{\alpha^2}|\psi_j(\alpha)|^2,
    \end{equation}
    so if $\lambda_m'(G',\alpha^c) =0$, solving for $\alpha^c$ gives
    $(\alpha^c)^2 = |\psi_j(\alpha^c)|^2/|\psi_i(\alpha^c)|^2$.
    Taking square roots gives critical points satisfying \eqref{pos_crit_point} and \eqref{neg_crit_point}.
    If $\alpha^c$ satisfies \eqref{pos_crit_point}, a straightforward computation shows that $\psi(\alpha^c) \in \ker(B_{ij}(\alpha^c))$ so $\psi(\alpha^c)$ is an eigenvector of $L^{\p P}(G)$ with eigenvalue $\lambda_m(G',\alpha^c)$. 

    On the other hand, if $\psi$ is a non-degenerate eigenvector of $L^{\p P}(G)$, the same computation shows that $\psi$ is an eigenvector of $L^{\p P}(G',\alpha_0)$, where $\alpha_0 = \psi_j/\psi_i$,
    and this $\alpha_0$ is a critical point of $\lambda_m(G',\alpha)$, as long as $\lambda_m(G',\alpha_0)$ is simple.

    Weyl interlacing shows that perturbation by $B_{ij}(\alpha)$ interlaces the spectra of $L^{\p P}(G)$ and $L^{\p P}(G',\alpha)$, so $|n-m|\leq 1$ where $n$ is the index of the eigenvalue in the spectrum of $L^{\p P}(G)$.
\end{proof}

It will be useful to temporarily treat $G'$ as our base graph, and investigate how the spectrum of $L^{\p P}(G',\alpha)$ changes as we add back the edge $(i,j)$ and remove the potentials.  If $\alpha > 0$, we have that $\lambda_m(G)\leq \lambda_m(G',\alpha)\leq \lambda_{m+1}(G)$. If $\alpha^c>0$ is a critical point of $\lambda_m(G',\alpha)$, Theorem \ref{thm:crit_points_and_efuns} shows that $\lambda_m(G',\alpha^c)$ must be an eigenvalue of $L^{\p P}$. Hence, we must have either $\lambda_m(G',\alpha^c) = \lambda_m(G)$ or $\lambda_m(G',\alpha^c) = \lambda_{m+1}(G)$. This means that if $\lambda(G',\alpha^c)$ is a local minimum, there is no shift in the eigenvalue index going from $G'$ to $G$ by adding $(i,j)$. On the other hand, if $\lambda_m(G',\alpha^c)$ is a local maximum, there is a shift of $+1$ in the index going from $G'$ to $G$. A similar argument can be made in the case that $\alpha^c < 0$, where the index shifts by $-1$ if $\lambda_{m}(G',\alpha^c)$ is a local minimum, and by $0$ if $\lambda_{m}(G',\alpha^c)$ is a local maximum. We have thus proven the following.

\begin{coro}\label{coro:spectral_shift}
    Let $\alpha^c \neq 0$ be a critical point of $\lambda_m(G',\alpha)$ and suppose that $\lambda_m(G',\alpha^c)$ is simple, so $\lambda_m(G',\alpha^c)$ is an eigenvalue of $L^{\p P}(G)$. Let $\Delta m_{ij}$ denote the shift in position of the eigenvalue in the spectrum going from $G'$ to $G$ by subtracting $w_{ij}B_{ij}(\alpha^c)$ from $L^{\p P}(G',\alpha)$. If $\alpha^c > 0$ we have 
    \begin{equation}\label{positive_alpha_spec_shift}
        \Delta m_{ij} = \begin{cases}
            +1, & \lambda_m(G',\alpha^c) \text{ is a local maximum}, \\
            0, &\lambda_m(G',\alpha^c) \text{ is a local minimum}.
        \end{cases}
    \end{equation}
    If $\alpha^c < 0$,
    \begin{equation}\label{negative_alpha_spec_shift}
        \Delta m_{ij} = \begin{cases}
            0, & \lambda_m(G',\alpha^c) \text{ is a local maximum}, \\
            -1, &\lambda_m(G',\alpha^c) \text{ is a local minimum}.
        \end{cases}
    \end{equation}
    Let $M_{ij} = 1$ if $\lambda_m(G',\alpha^c)$ is a local maximum, and $M_{ij} = 0$ otherwise. Then, we can combine \eqref{positive_alpha_spec_shift} and \eqref{negative_alpha_spec_shift} to write the more concise
    \begin{equation}\label{spectral_shift_formula}
        M_{ij} - \Delta m_{ij} = \begin{cases}
            0, & \alpha^c > 0 \\
            1, & \alpha^c < 0.
        \end{cases}
    \end{equation}
\end{coro}

\subsection{Upgrading to partitions}
So far we have only discussed the procedure for removing a single edge from $G$, but we can apply the perturbations in Theorem \ref{thm:crit_points_and_efuns} repeatedly, introducing a new parameter $\alpha_{ij}$ for each $(i,j) \in \p P$ we remove from $G$. Let $\R^{|\p P|}_{>0}$ consist of all vectors in $\R^{|\p P|}$ with strictly positive entries, and let $\alpha \in \R^{|\p P|}_{>0}$ have components $\alpha_{ij}$ for each $(i,j) \in \p P$. We restrict to $\alpha$ with strictly positive entries since these will be the points that correspond to eigenvectors of $L^{\p P}$. We define 
\begin{equation}
    L^{\p P}(P,\alpha) = L^{\p P}(G) + \sum_{(i,j) \in \p P}w_{ij}B_{ij}(\alpha_{ij}),
\end{equation}
where $L^{\p P}(P,\alpha)$ is the resulting operator when we have removed all of $\p P$ from $G$. We call the pair $(P,\alpha)$ a \textit{parameter-dependent partition} of $G$.

Since we have removed all of $\p P$, it follows that $L^{\p P}(P,\alpha)$ has a block diagonal structure, with each block corresponding to one connected component of $P$. If the connected components of $P$ are labeled $\{G_{k}\}_{k=1}^{\nu}$, we write
\begin{equation}
    L^{\p P}(P,\alpha) = \begin{pmatrix}
            L^{\p P}(G_1,\alpha) & & \\
            & \ddots & \\
            & & L^{\p P}(G_{\nu},\alpha)
        \end{pmatrix}
\end{equation}
where $L^{\p P}(G_k,\alpha)$ is the block corresponding to $G_k$. Note that $L^{\p P}(G_k,\alpha)$ is simply the usual Laplacian on $G_k$ with an added $\alpha$-dependent potential on vertices in $G_k$ that are incident to $\p P$. Hence, each $L^{\p P}(G_k,\alpha)$ is a generalized Laplacian \cite[Chapter 2]{biyikogluLaplacianEigenvectorsGraphs2007}.

For each $k=1,\dots,\nu$, let $\lambda_1(G_k,\alpha)$ denote the first eigenvalue of $L^{\p P}(G_k,\alpha)$. From \cite[Corollary 2.23]{biyikogluLaplacianEigenvectorsGraphs2007}, we know that $\lambda_1(G_k,\alpha)$ is a simple eigenvalue, and the corresponding eigenvector $f(G_k,\alpha)$ can be chosen to be strictly positive on $G_k$. We consider $f(G_k,\alpha)$ to be normalized so that
$\sum_{i \in V_k}|f_{i}(G_k,\alpha)|^2 = 1$.
When convenient we will consider $f(G_k,\alpha)$ to be a function on all of $V$ through extending by zero to the rest of $G$.

\begin{defn}\label{partition_energy}
    Given a parameter-dependent $\nu$-partition $(P,\alpha)$, we define the partition energy as
    \begin{equation}
        \Lambda(P,\alpha) = \max_{k=1,\dots,\nu}\lambda_1(G_k,\alpha).
    \end{equation}
    We say that $(P,\alpha)$ is an equipartition if 
    \begin{equation}
        \Lambda(P,\alpha) = \lambda_1(G_{1},\alpha) = \cdots = \lambda_1(G_{\nu},\alpha).
    \end{equation}
\end{defn}

It is not hard to verify that any parameter-dependent partition $(P,\alpha)$ that is locally minimal (with respect to variation of $\alpha$) must be an equipartition. This means that to find minimal partitions, we can limit our search to equipartitions. Given a partition $P$, define 
$\mathcal{E}_P = \{\alpha \in \R^{|\p P|}_{>0}: (P,\alpha) \text{ is an equipartition}\}$.

We find that non-degenerate eigenvectors of $L^{\p P}$ with nodal partition $P$ necessarily generate equipartitions.

\begin{lemma}\label{lem:evec_induces_equipartition}
    Suppose that $\psi$ is a non-degenerate eigenvector of $L^{\p P}$ with nodal partition $P$ and eigenvalue $\lambda_n$. Then, $\psi$ induces an equipartition $(P,\wtil{\alpha})$ by setting 
    \[\wtil{\alpha}_{ij} = \frac{\psi_j}{\psi_i},\]
    and the equipartition energy is given by $\Lambda(P,\wtil{\alpha}) = \lambda_n$.
\end{lemma}
Recall from the discussion at the end of Section \ref{part_lap_subsection} that $\psi$ having nodal partition $P$ implies that $\wtil{\alpha}_{ij}>0$ for all $(i,j) \in \p P$.
\begin{proof}
    This choice of $\wtil{\alpha}$ results in $\psi$ being in the kernel of our perturbation $B_{ij}(\wtil{\alpha}_{ij})$ for all $(i,j) \in \p P$. Hence, $\psi$ is an eigenvector of $L^{\p P}(P;\wtil{\alpha})$.
    Since $L^{\p P}(P;\wtil{\alpha})$ is block diagonal, if $\psi^{(k)} = \psi|_{G_k}$ extended by zero we find that $\psi^{(k)}$ must be an eigenvector of $L^{\p P}(G_k;\wtil{\alpha})$ with eigenvalue $\lambda_n$, as desired. Additionally, since $\psi$ is positive on each $G_k$ we see that $\psi^{(k)}$ is the ground state eigenvector on $G_k$, so $\lambda_n$ is the first eigenvalue. 
\end{proof}

\section{Proof of Theorem \ref{thm:main_result}}\label{sec:main_result}

Our main result shows that for certain partitions, $\mathcal{E}_P$ can be given a smooth structure that allows us to analyze critical points of the map $\Lambda:\mathcal{E}_P \to \R$ mapping $\alpha \mapsto \Lambda(P,\alpha)$. In \cite{berkolaikoStabilityNodalStructures2012}, the authors give an explicit embedding of $\mathcal{E}_P$ in $\R^{|\p P|}$ but we take a different approach here using transversality, inspired by the methods in \cite{berkolaikoCriticalPartitionsNodal2012}. 

Let $\wtil{\alpha} \in \mathcal{E}_P$. Fix $\ep>0$ such that $B_{\ep}(\wtil{\alpha})$ does not intersect any of the hyperplanes $\alpha_{ij} = 0$ and also such that the variation of $\Lambda(P,\alpha)$ over $B_{\ep}(\wtil{\alpha})$ is less than the minimum separation of distinct eigenvalues in the spectrum of $L^{\p P}(G)$. This second condition is what allows us to treat $\alpha \mapsto \Lambda(P,\alpha)$ as a smooth map when restricted to $\mathcal{E}_P$.

We define a map $\Phi:B_{\ep}(\wtil{\alpha}) \to \R^{\nu}$ as 
\begin{equation}\label{phi_map}
    \Phi(\alpha) = (\lambda_1(G_1,\alpha),\dots,\lambda_1(G_{\nu},\alpha)).
\end{equation}
It is clear that $\Phi$ is smooth since each eigenvalue $\lambda_1(G_k,\alpha)$ is simple. Let $\lap \sub \R^{\nu}$ be the diagonal
\begin{equation}\label{diagonal}
    \lap = \{(\lambda,\lambda,\dots,\lambda) \in \R^{\nu}: \lambda \in \R\} = \spa\{(1,1,\dots,1)\}.
\end{equation}
We denote the tangent space of $\R^{\nu}$ at $x$ by $T_x\R^{\nu}$ and the tangent space of $\lap$ at $x$ by $T_{x}\lap$. Note that $T_{x}\lap$ is one-dimensional and $T_{x}\lap = \lap$. We also let $d\Phi_{x}:T_{x}\R^{|\p P|} \to T_{\Phi(x)}\R^{\nu}$ denote the differential of $\Phi$ at $x$, which we identify with a $\nu \times |\p P|$ matrix. Note also that $\Phi^{-1}(\lap) = \mathcal{E}_P\cap B_{\ep}(\wtil{\alpha}).$

\begin{defn}[{\cite[Definition 8.1]{benedettiLecturesDifferentialTopology2021}}]
    We say that $\Phi$ is transversal to $\lap \sub \R^{\nu}$ if for any $\alpha \in \Phi^{-1}(\lap)$ with $\Phi(\alpha) = \zeta \in \R^{\nu}$, we have
    \[T_{\zeta}\R^{\nu} = T_{\zeta}\lap + d\Phi_{\alpha}(T_{\alpha}\R^{|\p P|}),\]
    where the addition denotes vector addition.
\end{defn}
Of course, we can identify the tangent space of $\R^n$ with itself so an equivalent condition is that
\[\R^{\nu} =T_{\zeta}\lap + d\Phi_{\alpha}(T_{\alpha}\R^{|\p P|}) =  \lap + d\Phi_{\alpha}(\R^{|\p P|}).\]
The key result we use is the following:
\begin{theorem}[{\cite[Theorem 8.2]{benedettiLecturesDifferentialTopology2021}}]\label{transvers_thm}
    If $\Phi:B_{\ep}(\wtil{\alpha}) \to \R^{\nu}$ is transversal to $\lap$, then $\Phi^{-1}(\lap)$ is a smooth submanifold of $\R^{|\p P|}$ with codimension $\nu - 1$. In other words, $\mathcal{E}_P\cap B_{\ep}(\wtil{\alpha})$ is a smooth submanifold of $\R^{|\p P|}$ of dimension $\eta = |\p P| - (\nu -1 )$.
\end{theorem}
Note that $\eta = |\p P| - (\nu -1)$ is the first Betti number (or cyclomatic number) of the multigraph $P^{G}$ defined in Section \ref{part_lap_subsection} and is always nonnegative \cite{hatcher2002algebraic}. Crucially, $\eta$ is also the number of independent cycles in $P^G$, or the minimum number of edges to remove from $P^G$ to form a tree graph.

\begin{figure}
    \centering
    \begin{tikzpicture}[x=0.75pt,y=0.75pt,yscale=-1,xscale=1]
    
    \draw [color={rgb, 255:red, 74; green, 144; blue, 226 }  ,draw opacity=1 ][line width=1.5]    (164.2,139.3) -- (173.75,98.5) ;
    \draw [color={rgb, 255:red, 245; green, 166; blue, 35 }  ,draw opacity=1 ][line width=1.5]    (130,155.5) .. controls (163.25,157.5) and (161.75,120.5) .. (205.25,130.5) ;
    \draw  [dash pattern={on 1.69pt off 2.76pt}][line width=1.5]  (125,137.3) .. controls (125,115.1) and (143,97.1) .. (165.2,97.1) .. controls (187.4,97.1) and (205.4,115.1) .. (205.4,137.3) .. controls (205.4,159.5) and (187.4,177.5) .. (165.2,177.5) .. controls (143,177.5) and (125,159.5) .. (125,137.3) -- cycle ;
    \draw [line width=1.5]  (40,231.3) -- (232,231.3)(59.2,58.5) -- (59.2,250.5) (225,226.3) -- (232,231.3) -- (225,236.3) (54.2,65.5) -- (59.2,58.5) -- (64.2,65.5)  ;
    \draw  [color={rgb, 255:red, 0; green, 0; blue, 0 }  ,draw opacity=1 ][fill={rgb, 255:red, 0; green, 0; blue, 0 }  ,fill opacity=1 ][line width=1.5]  (160.95,139.3) .. controls (160.95,137.51) and (162.41,136.05) .. (164.2,136.05) .. controls (165.99,136.05) and (167.45,137.51) .. (167.45,139.3) .. controls (167.45,141.09) and (165.99,142.55) .. (164.2,142.55) .. controls (162.41,142.55) and (160.95,141.09) .. (160.95,139.3) -- cycle ;
    \draw    (183.98,136) -- (215.25,164) ;
    \draw [shift={(181.75,134)}, rotate = 41.85] [fill={rgb, 255:red, 0; green, 0; blue, 0 }  ][line width=0.08]  [draw opacity=0] (6.25,-3) -- (0,0) -- (6.25,3) -- cycle    ;
    \draw [line width=1.5]  (371.5,230.8) -- (563.5,230.8)(390.7,58) -- (390.7,250) (556.5,225.8) -- (563.5,230.8) -- (556.5,235.8) (385.7,65) -- (390.7,58) -- (395.7,65)  ;
    \draw [line width=1.5]    (377.25,243.5) -- (554.75,66) ;
    \draw    (257,88.5) .. controls (296.2,59.1) and (316.91,58.99) .. (354.67,86.76) ;
    \draw [shift={(357,88.5)}, rotate = 216.93] [fill={rgb, 255:red, 0; green, 0; blue, 0 }  ][line width=0.08]  [draw opacity=0] (8.93,-4.29) -- (0,0) -- (8.93,4.29) -- cycle    ;
    \draw [color={rgb, 255:red, 245; green, 166; blue, 35 }  ,draw opacity=1 ][line width=1.5]    (410.75,210) -- (486.25,134.5) ;
    \draw [shift={(486.25,134.5)}, rotate = 135] [color={rgb, 255:red, 245; green, 166; blue, 35 }  ,draw opacity=1 ][line width=1.5]      (6.71,-6.71) .. controls (3.01,-6.71) and (0,-3.7) .. (0,0) .. controls (0,3.7) and (3.01,6.71) .. (6.71,6.71) ;
    \draw [shift={(410.75,210)}, rotate = 315] [color={rgb, 255:red, 245; green, 166; blue, 35 }  ,draw opacity=1 ][line width=1.5]      (6.71,-6.71) .. controls (3.01,-6.71) and (0,-3.7) .. (0,0) .. controls (0,3.7) and (3.01,6.71) .. (6.71,6.71) ;
    \draw  [color={rgb, 255:red, 208; green, 2; blue, 27 }  ,draw opacity=1 ][fill={rgb, 255:red, 208; green, 2; blue, 27 }  ,fill opacity=1 ] (188.25,128.63) .. controls (188.25,127.31) and (189.31,126.25) .. (190.63,126.25) .. controls (191.94,126.25) and (193,127.31) .. (193,128.63) .. controls (193,129.94) and (191.94,131) .. (190.63,131) .. controls (189.31,131) and (188.25,129.94) .. (188.25,128.63) -- cycle ;
    \draw [color={rgb, 255:red, 208; green, 2; blue, 27 }  ,draw opacity=1 ]   (193,128.63) -- (224.25,128.51) ;
    \draw [shift={(227.25,128.5)}, rotate = 179.79] [fill={rgb, 255:red, 208; green, 2; blue, 27 }  ,fill opacity=1 ][line width=0.08]  [draw opacity=0] (5.36,-2.57) -- (0,0) -- (5.36,2.57) -- cycle    ;
    \draw [color={rgb, 255:red, 208; green, 2; blue, 27 }  ,draw opacity=1 ]   (190.63,128.63) -- (190.74,95.5) ;
    \draw [shift={(190.75,92.5)}, rotate = 90.2] [fill={rgb, 255:red, 208; green, 2; blue, 27 }  ,fill opacity=1 ][line width=0.08]  [draw opacity=0] (5.36,-2.57) -- (0,0) -- (5.36,2.57) -- cycle    ;
    \draw  [color={rgb, 255:red, 208; green, 2; blue, 27 }  ,draw opacity=1 ][fill={rgb, 255:red, 208; green, 2; blue, 27 }  ,fill opacity=1 ] (461.75,156.63) .. controls (461.75,155.31) and (462.81,154.25) .. (464.13,154.25) .. controls (465.44,154.25) and (466.5,155.31) .. (466.5,156.63) .. controls (466.5,157.94) and (465.44,159) .. (464.13,159) .. controls (462.81,159) and (461.75,157.94) .. (461.75,156.63) -- cycle ;
    \draw [color={rgb, 255:red, 208; green, 2; blue, 27 }  ,draw opacity=1 ]   (464.13,156.63) -- (446.39,129.51) ;
    \draw [shift={(444.75,127)}, rotate = 56.81] [fill={rgb, 255:red, 208; green, 2; blue, 27 }  ,fill opacity=1 ][line width=0.08]  [draw opacity=0] (5.36,-2.57) -- (0,0) -- (5.36,2.57) -- cycle    ;
    
    \draw (217,237.9) node [anchor=north west][inner sep=0.75pt]    {$\alpha _{1}$};
    \draw (26,59.9) node [anchor=north west][inner sep=0.75pt]    {$\alpha _{2}$};
    \draw (159,142.9) node [anchor=north west][inner sep=0.75pt]  [color={rgb, 255:red, 0; green, 0; blue, 0 }  ,opacity=1 ]  {$\tilde{\alpha }$};
    \draw (153,104.4) node [anchor=north west][inner sep=0.75pt]  [color={rgb, 255:red, 74; green, 144; blue, 226 }  ,opacity=1 ]  {$\varepsilon $};
    \draw (216.5,162.9) node [anchor=north west][inner sep=0.75pt]  [color={rgb, 255:red, 245; green, 166; blue, 35 }  ,opacity=1 ]  {$\mathcal{E}_{P}$};
    \draw (124,35.9) node [anchor=north west][inner sep=0.75pt]    {$\mathbb{R}^{|\partial P|}$};
    \draw (299,47.4) node [anchor=north west][inner sep=0.75pt]    {$\Phi $};
    \draw (483,36.9) node [anchor=north west][inner sep=0.75pt]    {$\mathbb{R}^{\nu }$};
    \draw (542,87.9) node [anchor=north west][inner sep=0.75pt]    {$\Delta $};
    \draw (458.5,169.9) node [anchor=north west][inner sep=0.75pt]  [color={rgb, 255:red, 245; green, 166; blue, 35 }  ,opacity=1 ]  {$\Phi \left(\mathcal{E}_{P} \cap B_{\varepsilon }\left(\tilde{\alpha }\right)\right)$};
    \draw (222.5,127.9) node [anchor=north west][inner sep=0.75pt]  [color={rgb, 255:red, 208; green, 2; blue, 27 }  ,opacity=1 ]  {$v_{1}$};
    \draw (196.5,81.4) node [anchor=north west][inner sep=0.75pt]  [color={rgb, 255:red, 208; green, 2; blue, 27 }  ,opacity=1 ]  {$v_{2}$};
    \draw (412.5,105.9) node [anchor=north west][inner sep=0.75pt]  [color={rgb, 255:red, 208; green, 2; blue, 27 }  ,opacity=1 ]  {$d\Phi _{\alpha }( v_{2})$};
    \end{tikzpicture}
    \caption{A schematic of the transversality argument with $|\p P| = \nu = 2$. The matrix $d\Phi_{\alpha}$ may not have full rank at every $\alpha$, but the key is that the tangent space of $\Delta$ always complements the image of $d\Phi_{\alpha}$ to give all of $\R^{\nu}$.}
\end{figure}
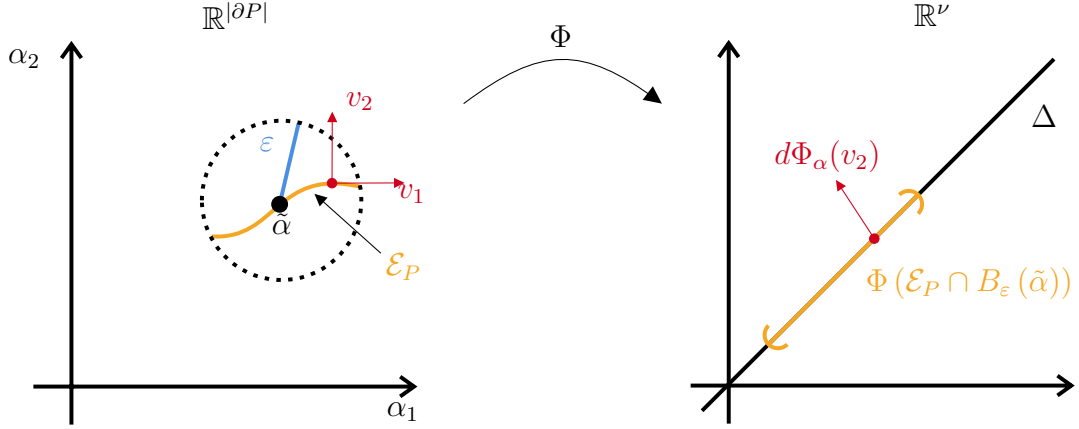

\begin{lemma}\label{phi_transverse}
    $\Phi:B_{\ep}(\wtil{\alpha}) \to \R^{\nu}$ as defined in \eqref{phi_map} is transversal to $\lap$.
\end{lemma}

\begin{proof}
    Suppose that $\alpha \in \Phi^{-1}(\lap)$.
    We can consider $d\Phi_{\alpha}$ as a $\nu \times |\p P|$ matrix. Recall that 
    $\R^{\nu} = d\Phi_{\alpha}(\R^{|\p P|}) \oplus \ker(d\Phi_{\alpha}^T)$.
    Each row of $(d\Phi_{\alpha})^T$ has the form 
    \[\left(\pderiv{\lambda_1(G_1,\alpha)}{\alpha_{ij}},\dots,\pderiv{\lambda_1(G_{\nu},\alpha)}{\alpha_{ij}}\right) \in \R^{\nu},\]
    but only two of these components are nonzero since each $\alpha_{ij}$ only appears in the components containing vertices $i$ and $j$. Explicitly computing the derivatives as in the proof of Theorem \ref{thm:crit_points_and_efuns} gives
    \[\pderiv{\lambda_1(G_k,\alpha)}{\alpha_{ij}} = \begin{cases}
        |f_i(G_k,\alpha)|^2, & k = s(i) \\
        -\frac{1}{\alpha_{ij}^2}|f_j(G_k,\alpha)|^2, & k = s(j) \\
        0, & \text{else},
    \end{cases}\]
    where $s:V \to \{1,2,\dots,\nu\}$ maps a vertex $i$ to the index of the subdomain in which it lives. It follows then that $u \in \R^{\nu}$ is in the kernel of $(d\Phi_{\alpha})^T$ if and only if 
    \begin{equation}\label{kernel_criterion}
        |f_i(G_{s(i)},\alpha)|^2u_{s(i)} -\frac{1}{\alpha_{ij}^2}|f_j(G_{s(j)},\alpha)|^2u_{s(j)} = 0
    \end{equation}
    for all $(i,j) \in \p P$. Since the partition graph $P^G$ is connected, the kernel is at most one-dimensional. To be more explicit, if $u \in \ker (d\Phi_{\alpha}^T)$ with $u_{k} \neq 0$, \eqref{kernel_criterion} shows that $u_{\ell}$ is uniquely determined for all $\ell$ where there is some edge in $\p P$ connecting $G_{k}$ with $G_{\ell}$. Since the multigraph $P^G$ is connected, continuing this process iteratively determines all of $u$ uniquely.
    
    If $\ker(d\Phi_{\alpha}^T) = 0$, we are done since it implies the column space of $d\Phi_{\alpha}$ is all of $\R^{\nu}$. Suppose now that $\ker(d\Phi_{\alpha}^T)$ is spanned by some $u \in \R^{\nu}$. Then by \eqref{kernel_criterion} all components of $u$ must have the same sign. But then $u$ cannot be orthogonal to $(1,1,\dots,1) \in \R^{\nu}$, which spans $\lap$. Hence, $\lap$ is not contained in $d\Phi_{\alpha}(\R^{|\p P|})$, so we have 
    $\R^{\nu} = \lap + d\Phi_{\alpha}(\R^{|\p P|})$,
    as desired.
\end{proof}

We are now finally ready to state and prove our main result. There are two cases to consider: $\eta=0$ and $\eta>0$. When $\eta = 0$, $P$ is necessarily a tree partition. Berkolaiko, Raz and Smilansky showed that any tree equipartition is unique and must correspond to the nodal partition of a Courant-sharp Laplacian eigenvector \cite[Theorem 4.3]{berkolaikoCriticalPartitionsNodal2012}. Any tree partition is bipartite, so it must also be the nodal partition of a Courant-sharp eigenvector of the partition Laplacian $L^{\p P}$ (using the unitary equivalence described in Section \ref{part_lap_subsection}). We hence immediately have the following:
\begin{theorem}\label{tree_partition}
    Let $(P,\wtil{\alpha})$ be a $\nu$-equipartition with $\eta = |\p P| - (\nu -1 ) =0$. Then $(P,\wtil{\alpha})$ is the unique equipartition with respect to $\alpha$, and $P$ is the nodal partition of a Courant-sharp eigenvector of $L^{\p P}$ with eigenvalue $\Lambda(P,\wtil{\alpha})$.
\end{theorem}

The main new development here is that we can use the partition Laplacian to remove the bipartite assumption in \cite[Theorem 4.4]{berkolaikoCriticalPartitionsNodal2012}.

\begin{theorem}\label{thm:main_theorem}
    Let $(P,\wtil{\alpha})$ be a $\nu$-equipartition with $\wtil{\alpha}_e > 0$ for all $e \in \p P$ and suppose $\eta = |\p P| - (\nu-1)>0$. Then $\mathcal{E}_P$ forms a smooth $\eta$-dimensional submanifold of $\R^{|\p P|}$ in a neighborhood of $\wtil{\alpha}$. Furthermore, $\alpha$ is a critical point of $\Lambda$ on $\mathcal{E}_P$ if and only if it corresponds to a non-degenerate eigenvector $\psi$ of $L^{\p P}(G)$ with nodal partition $P$ and eigenvalue $\Lambda(P,\alpha)$. Moreover, 
    \[\text{Mor}(\Lambda(P,\alpha)) = \delta(\psi),\]
    where Mor$(\Lambda(P,\alpha))$ is the Morse index of the critical point, and $\delta(\psi)$ is the nodal deficiency of $\psi$ as an eigenvector of $L^{\p P}(G)$.
\end{theorem}
In other words, the stability of any critical point of $\Lambda$ is determined entirely by the nodal deficiency of the corresponding non-degenerate eigenvector of $L^{\p P}(G)$. In particular, this tells us that an equipartition is locally minimal if and only if it is the nodal partition of a Courant-sharp eigenvector of $L^{\p P}(G)$. 

\begin{proof}
    Since $\Phi:B_{\ep}(\wtil{\alpha}) \to \R^{\nu}$ is transversal to the diagonal $\lap$, we have that $\Phi^{-1}(\lap)$ is a smooth submanifold of $B_{\ep}(\wtil{\alpha})$ of dimension $|\p P| - (\nu -1)$, which is strictly positive by assumption. Recall that $\Phi^{-1}(\lap) = \mathcal{E}_P \cap B_{\ep}(\wtil{\alpha})$ and we find that $\mathcal{E}_P$ is a smooth submanifold of $\R^{|\p P|}$ in a neighborhood of $\wtil{\alpha}$.

    Suppose we have some critical point $\alpha^c$ of the map $\alpha \mapsto \Lambda(P,\alpha)$ defined in a neighborhood of $\wtil{\alpha}$ in $\mathcal{E}_P$. We will use Lagrange multipliers, where we consider critical points of the map $\alpha \mapsto \lambda_1(G_1,\alpha)$ with the $\nu-1$ constraints
    \[\lambda_1(G_1,\alpha) = \cdots = \lambda_1(G_{\nu},\alpha).\]
    Let $\{\zeta_k\}_{k=2}^{\nu}$ denote our Lagrange multipliers and define our Lagrangian function by
    \[\mathcal{L}(\alpha,\zeta) = \lambda_1(G_1,\alpha) + \sum_{k=2}^{\nu}\zeta_k(\lambda_1(G_k,\alpha) - \lambda_1(G_1,\alpha)).\]
    Alternatively we can collect all eigenvalues corresponding to the same domain and write 
    \begin{equation}\label{lagrangian_function}
        \mathcal{L}(\alpha,\zeta) = \sum_{k=1}^{\nu}c_k\lambda_{1}(G_k,\alpha)
    \end{equation}
    where each $c_k$ is some linear function of $\zeta$. To be explicit, $c_1 = 1 - \zeta_2 - \dots -\zeta_{\nu}$ and $c_k = \zeta_k$ for $k\geq 2$. Note that $\sum_{k}c_k = 1$. 
    
    For each $(i,j) \in \p P$, the only two terms of \eqref{lagrangian_function} that depend on $\alpha_{ij}$ are the eigenvalues corresponding to the respective subdomains of $P$ to which $i$ and $j$ belong. Hence, if we take the derivative in $\alpha_{ij}$ and set equal to zero, we obtain 
    \begin{equation}
        c_{s(i)}\pderiv{\lambda_1(G_{s(i)},\alpha)}{\alpha_{ij}} + c_{s(j)}\pderiv{\lambda_1(G_{s(j)},\alpha)}{\alpha_{ij}}=0,
    \end{equation}
    where $s$ maps any vertex $i$ to the index of the subdomain in which it lives. If $f(G_k,\alpha)$ is the ground state eigenvector on $G_k$, using first order perturbation theory like in the proof of Theorem \ref{thm:crit_points_and_efuns}, we find that our critical point $\alpha^c$ must satisfy
    \begin{equation}\label{crit_points_must_satisfy}
        c_{s(i)}|f_{i}(G_{s(i)},\alpha^c)|^2 - \frac{c_{s(j)}}{(\alpha^c_{ij})^2}|f_{j}(G_{s(j)},\alpha^c)|^2 = 0.
    \end{equation}
    This immediately implies that $c_k \neq 0$ for all $k$ since our graph is connected, so $c_k>0$ for all $k$.

    We define 
    \begin{equation}
        \psi = \sum_{k=1}^{\nu}\sqrt{c_k}f(G_k,\alpha^c)
    \end{equation}
    where we consider the ground states on each $G_k$ to be extended by zero to the rest of $G$. Clearly $\psi$ is an eigenvector of $L^{\p P}(P,\alpha^c)$ with eigenvalue $\Lambda(P,\alpha^c)$, and we note that $\psi_i = \sqrt{c_{\chi(i)}}f_i(G_{s(i)},\alpha^c)$.
    Plugging $\psi$ into \eqref{crit_points_must_satisfy} gives 
    \begin{equation}\label{psi_must_satisfy}
        \psi_i^2 - \frac{1}{(\alpha^c_{ij})^2}\psi_j^2 = 0,
    \end{equation}
    or, alternatively, that $\alpha_{ij}^c = \frac{\psi_j}{\psi_i} >0$.
    Since \eqref{psi_must_satisfy} implies that $\psi \in \ker(B_{ij}(\alpha_{ij}^c))$ for all $(i,j) \in \p P$, it follows that $\psi$ must also be an eigenvector of $L^{\p P}(G)$ with eigenvalue $\Lambda(P,\alpha^c)$. In addition, $\psi$ is strictly positive and hence must have nodal partition $P$.

    For the other direction, assume that $\psi$ is a non-degenerate eigenvector of $L^{\p P}(G)$ with nodal partition $P$. Then define 
    $\alpha_{ij} = \psi_j/\psi_i$ 
    for each $(i,j) \in \p P$ and $c_k = \sum_{i \in V_k}\psi_i^2$.
    Let
    \[f(G_k,\alpha) = \frac{\psi|_{G_k}}{\sqrt{c_k}}\]
    and a simple computation shows that $f(G_k,\alpha)$ is a normalized eigenvector of $L^{\p P}(G_k,\alpha)$. Since $f(G_k,\alpha)$ has a constant sign, it must correspond to the first eigenvector. Plugging in $f(G_k,\alpha)$ to \eqref{crit_points_must_satisfy} shows that $\alpha$ is a critical point of $\Lambda$ on $\mathcal{E}_P$.

    We now discuss the stability of these critical points. Let $(P,\alpha^c)$ be a critical equipartition.
    Beginning with our base graph $G$, we remove $\eta >0$ edges from $\p P$, introducing a new potential parameter $\xi_{ij}$ in the usual way for each edge removed. Since $\eta$ is the first Betti number of $P^{G}$, we can choose these edges to remove so that $P$ is now a tree partition of the resultant graph $G_{\xi}$. Let $R \sub \p P$ denote the edges removed in this construction. 

    Let $\psi$ be the eigenvector of $L^{\p P}(G)$ associated to $(P,\alpha^c)$ and set $\xi_e = \alpha^c_e$ for all $e \in R$. Since $\psi \in \ker(B_{ij}(\xi_{ij}))$ for all $(i,j) \in R$, $\psi$ is an eigenvector of $L^{\p P}(G_{\xi})$ with eigenvalue $\Lambda(P,\alpha^c)$ and nodal partition $P$. Thus by Lemma \ref{lem:evec_induces_equipartition}, $\psi$ induces a tree equipartition of $G_{\xi}$. Tree equipartitions are unique by Theorem \ref{tree_partition}, so $\psi$ is a Courant-sharp eigenvector of $L^{\p P}(G_{\xi})$. 

    \begin{figure}
        \begin{subfigure}[c]{0.4\textwidth}
            \centering
            \begin{tikzpicture}[x=0.75pt,y=0.75pt,yscale=-1,xscale=1]
        
        \draw [color={rgb, 255:red, 0; green, 0; blue, 0 }  ,draw opacity=1 ][line width=1.5]    (233.61,140.63) -- (248.43,197.35) ;
        \draw [color={rgb, 255:red, 0; green, 0; blue, 0 }  ,draw opacity=1 ][line width=1.5]    (293.9,165.67) -- (252,198.88) ;
        \draw [color={rgb, 255:red, 0; green, 0; blue, 0 }  ,draw opacity=1 ][line width=1.5]    (294.92,160.56) -- (237.7,136.54) ;
        \draw  [line width=1.5]  (212.15,113.55) .. controls (212.15,111.29) and (213.98,109.46) .. (216.24,109.46) .. controls (218.49,109.46) and (220.32,111.29) .. (220.32,113.55) .. controls (220.32,115.81) and (218.49,117.64) .. (216.24,117.64) .. controls (213.98,117.64) and (212.15,115.81) .. (212.15,113.55) -- cycle ;
        \draw  [line width=1.5]  (247.92,111) .. controls (247.92,108.74) and (249.75,106.91) .. (252,106.91) .. controls (254.26,106.91) and (256.09,108.74) .. (256.09,111) .. controls (256.09,113.25) and (254.26,115.08) .. (252,115.08) .. controls (249.75,115.08) and (247.92,113.25) .. (247.92,111) -- cycle ;
        \draw  [line width=1.5]  (229.52,136.54) .. controls (229.52,134.29) and (231.35,132.46) .. (233.61,132.46) .. controls (235.87,132.46) and (237.7,134.29) .. (237.7,136.54) .. controls (237.7,138.8) and (235.87,140.63) .. (233.61,140.63) .. controls (231.35,140.63) and (229.52,138.8) .. (229.52,136.54) -- cycle ;
        \draw  [line width=1.5]  (244.34,201.44) .. controls (244.34,199.18) and (246.17,197.35) .. (248.43,197.35) .. controls (250.68,197.35) and (252.51,199.18) .. (252.51,201.44) .. controls (252.51,203.69) and (250.68,205.52) .. (248.43,205.52) .. controls (246.17,205.52) and (244.34,203.69) .. (244.34,201.44) -- cycle ;
        \draw  [line width=1.5]  (223.9,231.58) .. controls (223.9,229.33) and (225.73,227.5) .. (227.99,227.5) .. controls (230.24,227.5) and (232.08,229.33) .. (232.08,231.58) .. controls (232.08,233.84) and (230.24,235.67) .. (227.99,235.67) .. controls (225.73,235.67) and (223.9,233.84) .. (223.9,231.58) -- cycle ;
        \draw  [line width=1.5]  (267.33,228.52) .. controls (267.33,226.26) and (269.16,224.43) .. (271.42,224.43) .. controls (273.68,224.43) and (275.51,226.26) .. (275.51,228.52) .. controls (275.51,230.78) and (273.68,232.61) .. (271.42,232.61) .. controls (269.16,232.61) and (267.33,230.78) .. (267.33,228.52) -- cycle ;
        \draw  [line width=1.5]  (293.39,163.11) .. controls (293.39,160.86) and (295.22,159.03) .. (297.48,159.03) .. controls (299.74,159.03) and (301.57,160.86) .. (301.57,163.11) .. controls (301.57,165.37) and (299.74,167.2) .. (297.48,167.2) .. controls (295.22,167.2) and (293.39,165.37) .. (293.39,163.11) -- cycle ;
        \draw  [line width=1.5]  (314.34,131.95) .. controls (314.34,129.69) and (316.17,127.86) .. (318.43,127.86) .. controls (320.69,127.86) and (322.52,129.69) .. (322.52,131.95) .. controls (322.52,134.2) and (320.69,136.03) .. (318.43,136.03) .. controls (316.17,136.03) and (314.34,134.2) .. (314.34,131.95) -- cycle ;
        \draw  [line width=1.5]  (332.22,163.11) .. controls (332.22,160.86) and (334.05,159.03) .. (336.31,159.03) .. controls (338.57,159.03) and (340.4,160.86) .. (340.4,163.11) .. controls (340.4,165.37) and (338.57,167.2) .. (336.31,167.2) .. controls (334.05,167.2) and (332.22,165.37) .. (332.22,163.11) -- cycle ;
        \draw [line width=1.5]    (251.49,203.99) -- (268.86,224.94) ;
        \draw [line width=1.5]    (245.87,205.01) -- (227.99,227.5) ;
        \draw [line width=1.5]    (267.33,228.52) -- (232.08,231.58) ;
        \draw [line width=1.5]    (219.3,116.11) -- (231.05,133.99) ;
        \draw [line width=1.5]    (220.32,113.55) -- (247.92,111) ;
        \draw [line width=1.5]    (249.45,113.55) -- (236.67,133.99) ;
        \draw [line width=1.5]    (315.87,134.5) -- (297.48,159.03) ;
        \draw [line width=1.5]    (332.22,163.11) -- (301.06,162.09) ;
        \draw [line width=1.5]    (320.98,135.01) -- (336.31,159.03) ;
        \draw  [color={rgb, 255:red, 74; green, 144; blue, 226 }  ,draw opacity=1 ][dash pattern={on 1.69pt off 2.76pt}][line width=1.5]  (278.32,150.9) .. controls (278.32,135.32) and (295.6,122.69) .. (316.91,122.69) .. controls (338.22,122.69) and (355.5,135.32) .. (355.5,150.9) .. controls (355.5,166.48) and (338.22,179.12) .. (316.91,179.12) .. controls (295.6,179.12) and (278.32,166.48) .. (278.32,150.9) -- cycle ;
        \draw  [color={rgb, 255:red, 74; green, 144; blue, 226 }  ,draw opacity=1 ][dash pattern={on 1.69pt off 2.76pt}][line width=1.5]  (195.85,121.28) .. controls (195.85,105.7) and (213.13,93.07) .. (234.44,93.07) .. controls (255.75,93.07) and (273.03,105.7) .. (273.03,121.28) .. controls (273.03,136.87) and (255.75,149.5) .. (234.44,149.5) .. controls (213.13,149.5) and (195.85,136.87) .. (195.85,121.28) -- cycle ;
        \draw  [color={rgb, 255:red, 74; green, 144; blue, 226 }  ,draw opacity=1 ][dash pattern={on 1.69pt off 2.76pt}][line width=1.5]  (208.63,218.85) .. controls (208.63,203.27) and (225.91,190.64) .. (247.22,190.64) .. controls (268.53,190.64) and (285.81,203.27) .. (285.81,218.85) .. controls (285.81,234.43) and (268.53,247.07) .. (247.22,247.07) .. controls (225.91,247.07) and (208.63,234.43) .. (208.63,218.85) -- cycle ;
        
        \draw (296.4,191) node [anchor=north west][inner sep=0.75pt]    {$G$};
        \end{tikzpicture}
        \end{subfigure}
        \begin{subfigure}[c]{0.4\textwidth}
           \centering
           \begin{tikzpicture}[x=0.75pt,y=0.75pt,yscale=-1,xscale=1]

\draw [color={rgb, 255:red, 0; green, 0; blue, 0 }  ,draw opacity=1 ][line width=1.5]    (253.61,160.63) -- (268.43,217.35) ;
\draw [color={rgb, 255:red, 0; green, 0; blue, 0 }  ,draw opacity=1 ][line width=1.5]    (314.92,180.56) -- (257.7,156.54) ;
\draw  [line width=1.5]  (232.15,133.55) .. controls (232.15,131.29) and (233.98,129.46) .. (236.24,129.46) .. controls (238.49,129.46) and (240.32,131.29) .. (240.32,133.55) .. controls (240.32,135.81) and (238.49,137.64) .. (236.24,137.64) .. controls (233.98,137.64) and (232.15,135.81) .. (232.15,133.55) -- cycle ;
\draw  [line width=1.5]  (267.92,131) .. controls (267.92,128.74) and (269.75,126.91) .. (272,126.91) .. controls (274.26,126.91) and (276.09,128.74) .. (276.09,131) .. controls (276.09,133.25) and (274.26,135.08) .. (272,135.08) .. controls (269.75,135.08) and (267.92,133.25) .. (267.92,131) -- cycle ;
\draw  [line width=1.5]  (249.52,156.54) .. controls (249.52,154.29) and (251.35,152.46) .. (253.61,152.46) .. controls (255.87,152.46) and (257.7,154.29) .. (257.7,156.54) .. controls (257.7,158.8) and (255.87,160.63) .. (253.61,160.63) .. controls (251.35,160.63) and (249.52,158.8) .. (249.52,156.54) -- cycle ;
\draw  [line width=1.5]  (264.34,221.44) .. controls (264.34,219.18) and (266.17,217.35) .. (268.43,217.35) .. controls (270.68,217.35) and (272.51,219.18) .. (272.51,221.44) .. controls (272.51,223.69) and (270.68,225.52) .. (268.43,225.52) .. controls (266.17,225.52) and (264.34,223.69) .. (264.34,221.44) -- cycle ;
\draw  [line width=1.5]  (243.9,251.58) .. controls (243.9,249.33) and (245.73,247.5) .. (247.99,247.5) .. controls (250.24,247.5) and (252.08,249.33) .. (252.08,251.58) .. controls (252.08,253.84) and (250.24,255.67) .. (247.99,255.67) .. controls (245.73,255.67) and (243.9,253.84) .. (243.9,251.58) -- cycle ;
\draw  [line width=1.5]  (287.33,248.52) .. controls (287.33,246.26) and (289.16,244.43) .. (291.42,244.43) .. controls (293.68,244.43) and (295.51,246.26) .. (295.51,248.52) .. controls (295.51,250.78) and (293.68,252.61) .. (291.42,252.61) .. controls (289.16,252.61) and (287.33,250.78) .. (287.33,248.52) -- cycle ;
\draw  [line width=1.5]  (313.39,183.11) .. controls (313.39,180.86) and (315.22,179.03) .. (317.48,179.03) .. controls (319.74,179.03) and (321.57,180.86) .. (321.57,183.11) .. controls (321.57,185.37) and (319.74,187.2) .. (317.48,187.2) .. controls (315.22,187.2) and (313.39,185.37) .. (313.39,183.11) -- cycle ;
\draw  [line width=1.5]  (334.34,151.95) .. controls (334.34,149.69) and (336.17,147.86) .. (338.43,147.86) .. controls (340.69,147.86) and (342.52,149.69) .. (342.52,151.95) .. controls (342.52,154.2) and (340.69,156.03) .. (338.43,156.03) .. controls (336.17,156.03) and (334.34,154.2) .. (334.34,151.95) -- cycle ;
\draw  [line width=1.5]  (352.22,183.11) .. controls (352.22,180.86) and (354.05,179.03) .. (356.31,179.03) .. controls (358.57,179.03) and (360.4,180.86) .. (360.4,183.11) .. controls (360.4,185.37) and (358.57,187.2) .. (356.31,187.2) .. controls (354.05,187.2) and (352.22,185.37) .. (352.22,183.11) -- cycle ;
\draw [line width=1.5]    (271.49,223.99) -- (288.86,244.94) ;
\draw [line width=1.5]    (265.87,225.01) -- (247.99,247.5) ;
\draw [line width=1.5]    (287.33,248.52) -- (252.08,251.58) ;
\draw [line width=1.5]    (239.3,136.11) -- (251.05,153.99) ;
\draw [line width=1.5]    (240.32,133.55) -- (267.92,131) ;
\draw [line width=1.5]    (269.45,133.55) -- (256.67,153.99) ;
\draw [line width=1.5]    (335.87,154.5) -- (317.48,179.03) ;
\draw [line width=1.5]    (352.22,183.11) -- (321.06,182.09) ;
\draw [line width=1.5]    (340.98,155.01) -- (356.31,179.03) ;
\draw  [color={rgb, 255:red, 74; green, 144; blue, 226 }  ,draw opacity=1 ][dash pattern={on 1.69pt off 2.76pt}][line width=1.5]  (298.32,170.9) .. controls (298.32,155.32) and (315.6,142.69) .. (336.91,142.69) .. controls (358.22,142.69) and (375.5,155.32) .. (375.5,170.9) .. controls (375.5,186.48) and (358.22,199.12) .. (336.91,199.12) .. controls (315.6,199.12) and (298.32,186.48) .. (298.32,170.9) -- cycle ;
\draw  [color={rgb, 255:red, 74; green, 144; blue, 226 }  ,draw opacity=1 ][dash pattern={on 1.69pt off 2.76pt}][line width=1.5]  (215.85,141.28) .. controls (215.85,125.7) and (233.13,113.07) .. (254.44,113.07) .. controls (275.75,113.07) and (293.03,125.7) .. (293.03,141.28) .. controls (293.03,156.87) and (275.75,169.5) .. (254.44,169.5) .. controls (233.13,169.5) and (215.85,156.87) .. (215.85,141.28) -- cycle ;
\draw  [color={rgb, 255:red, 74; green, 144; blue, 226 }  ,draw opacity=1 ][dash pattern={on 1.69pt off 2.76pt}][line width=1.5]  (228.63,238.85) .. controls (228.63,223.27) and (245.91,210.64) .. (267.22,210.64) .. controls (288.53,210.64) and (305.81,223.27) .. (305.81,238.85) .. controls (305.81,254.43) and (288.53,267.07) .. (267.22,267.07) .. controls (245.91,267.07) and (228.63,254.43) .. (228.63,238.85) -- cycle ;

\draw (316.4,212.6) node [anchor=north west][inner sep=0.75pt]    {$G_{\xi }$};
\end{tikzpicture}
        \end{subfigure}
        \caption{The base graph $G$ on the left with some critical equipartition $(P,\alpha^c)$. This is the nodal partition of some eigenvector $\psi$ of $L^{\p P}$. By removing one edge, $\psi$ induces a critical equipartition of the modified graph $G_{\xi}$. This is a tree partition, and is hence unique, so $\psi$ must be Courant-sharp as an eigenvector on $G_{\xi}$.}
    \end{figure}
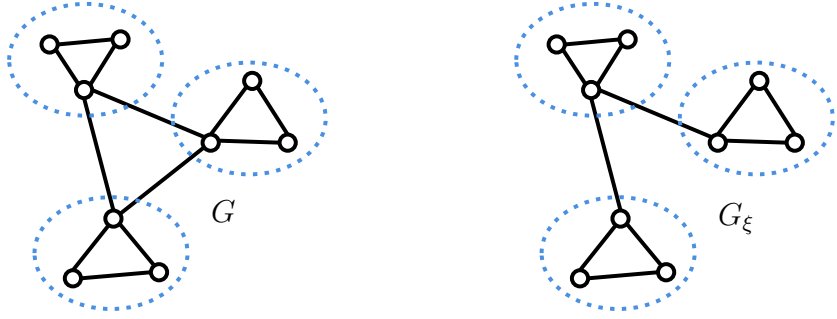

    If we start with $G_{\xi}$ and successively add back the edges of $R$, Corollary \ref{coro:spectral_shift} implies that the total shift $\Delta m$ in the index of the eigenvalue $\Lambda(P,\alpha^c)$ is equal to the number of edges $(i,j) \in R$ along which $\Lambda(P,\alpha^c)$ is a local maximum. Here we are using \eqref{spectral_shift_formula} to obtain 
    \[\Delta m = \sum_{(i,j) \in R}M_{ij}.\]
    Since $\psi(G_{\xi})$ is Courant-sharp, $\Lambda(P,\alpha^c)$ originally has index $\nu$. Hence, the final index $n$ of $\psi$ is given by 
    \[n = \nu + \sum_{(i,j) \in R}M_{ij}.\]
    Alternatively, the nodal deficiency $n - \nu$ of $\psi$ is equal to the number of directions in which $\Lambda(P,\alpha^c)$ is approached as a local maximum, which is exactly the Morse index of the critical point $\text{Mor}(\Lambda(P,\alpha^c)).$ Hence, $\text{Mor}(\Lambda(P,\alpha^c)) = \delta(\psi)$
    as desired.
\end{proof}

\section{Switching Equivalence and the Partition Energy}\label{sec:lower_bound}

Theorem \ref{thm:main_theorem} provides a characterization of local minimal partitions, but it requires us to fix an underlying combinatorial partition $P$. The main result of this section, Theorem \ref{thm:homology_result}, allows us to compare different partitions. As an immediate corollary, we find that Courant-sharp Laplacian eigenvectors induce partitions that are globally minimal.
The methods in this section are inspired by the results of Berkolaiko, Canzani, Cox, and Marzuola in \cite{berkolaikoHomologySpectralMinimal2024}.

Let $G = (V,E,w)$ be a finite, simple, weighted and connected graph. Let $\Gamma \sub E$. Define a signature $\sigma^{\Gamma}$ on $G$ by 
\begin{equation}
\sigma^{\Gamma}_{ij} = \begin{cases}
    -1, & (i,j) \in \Gamma \\
    1, & (i,j) \in E \setminus \Gamma.
\end{cases}
\end{equation}
For ease of notation, we let $L^{\Gamma}$ denote the signed Laplacian on the signed graph $(G,\sigma^{\Gamma})$. We already know from Section \ref{sec:background} that if $\sigma^{\Gamma_1}$ and $\sigma^{\Gamma_2}$ are switching equivalent, then $L^{\Gamma_1}$ and $L^{\Gamma_2}$ are unitarily equivalent via some map $\Phi:\R^{|V|} \to \R^{|V|}$.

\begin{lemma}\label{switching_equiv_crit}
    Let $\Gamma_1,\Gamma_2 \sub E$. Then $\sigma^{\Gamma_1}$ is switching equivalent to $\sigma^{\Gamma_2}$ if and only if $\Gamma_1 \Delta \Gamma_2$ is the boundary set of a bipartite partition.
\end{lemma}

\begin{proof}
    First note that $\sigma^{\Gamma_1\Delta\Gamma_2} = \sigma^{\Gamma_1}\sigma^{\Gamma_2}$ since 
    \begin{align}(\sigma^{\Gamma_1}\sigma^{\Gamma_2})_{ij} &= \begin{cases}
        -1, &(i,j) \in \Gamma_1\Delta\Gamma_2 \\
        1, & \text{else}
    \end{cases} \\
    &= \sigma^{\Gamma_1\Delta\Gamma_2}_{ij}.
    \end{align}
    Assume that $\Gamma_1\Delta \Gamma_2$ is the boundary set of a bipartite partition. By Lemma \ref{harary}, we know that $\sigma^{\Gamma_1 \Delta \Gamma_2}$ is balanced, and balanced sets are switching equivalent to an all-positive signature. Hence, there is some $\tau:V \to \{-1,+1\}$ such that $1 = \tau_i\sigma^{\Gamma_1\Delta\Gamma_2}_{ij}\tau_j$ for all $(i,j) \in E$. Expanding gives 
    $1 = \tau_i\sigma^{\Gamma_1\Delta\Gamma_2}_{ij}\tau_j = \tau_i\sigma^{\Gamma_1}_{ij}\sigma^{\Gamma_2}_{ij}\tau_j$
    and multiplying both sides by $\sigma_{ij}^{\Gamma_2}$ gives 
    \[\sigma^{\Gamma_2}_{ij} = \tau_i\sigma_{ij}^{\Gamma_1}\tau_j,\]
    so $\sigma^{\Gamma_1}$ is switching equivalent to $\sigma^{\Gamma_2}$.

    Using the same argument in the opposite direction gives that switching equivalence of $\sigma^{\Gamma_1}$ and $\sigma^{\Gamma_2}$ implies $\Gamma_1 \Delta \Gamma_2$ is the boundary set of a bipartite partition.
\end{proof}

We now show that the nodal partition of any eigenvector of $L^{\Gamma}$ has a boundary $\p P$ that is switching equivalent to $\Gamma$ itself. 
\begin{lemma}\label{lem:nodal_part_switching_equiv}
    Let $\psi$ be a non-degenerate
    eigenvector of $L^{\Gamma}$ with nodal partition $P$. Then $\sigma^{\p P}$ is switching equivalent to $\sigma^{\Gamma}$.
\end{lemma}
\begin{proof}
    By Lemma \ref{switching_equiv_crit} it is sufficient to show that $\p P \Delta \Gamma$ is the boundary set of a bipartite partition. We show this by proving that $\psi$ must change sign across each edge in $\p P \Delta \Gamma$, so the bipartite structure can be induced by the sign of $\psi$.

    Let $(i,j) \in \p P \Delta \Gamma$, so there are two cases to consider. If $(i,j) \in \p P \setminus \Gamma$, $(i,j)$ is in the nodal partition of $\psi$ so $\psi_i \sigma^{\Gamma}_{ij}\psi_j<0$. Since $\sigma_{ij}^{\Gamma} = 1$, we have $\psi_i\psi_j<0$. Similarly, if $(i,j) \in \Gamma \setminus \p P$, we have $\psi_i\sigma_{ij}^{\Gamma}\psi_j >0$. Since $\sigma_{ij}^{\Gamma} = -1$, we must have that $\psi_i\psi_j<0$ as desired.

    These are the only edges on which $\psi$ changes sign. If $(i,j) \in E$ is such that $\psi_i\psi_j<0$, then either $(i,j) \in \p P \setminus \Gamma$ or $(i,j) \in \Gamma \setminus \p P$. Hence, $\p P \Delta \Gamma$ is the boundary set of a bipartite partition.
\end{proof}

\begin{propo}\label{non-degen_evec_induces_equip}
    Let $\Gamma \sub E$ and let $\psi$ be a non-degenerate eigenvector of $L^{\Gamma}$ with corresponding eigenvalue $\lambda_{k}(\Gamma)$. Then the nodal partition $P$ of $\psi$ induces a critical equipartition $(P,\alpha)$ with energy $\lambda_k(\Gamma)$. 
\end{propo}

\begin{proof}
    Given $\Gamma \sub E$ and an eigenvector $\psi$ of $L^{\Gamma}$ with corresponding eigenvalue $\lambda_k(\Gamma)$, there exists, by Lemma \ref{lem:nodal_part_switching_equiv},  some unitary map $\Phi$ such that $\Phi \psi$ is an eigenvector of $L^{\p P}$. With Lemma \ref{lem:evec_induces_equipartition}, we can construct an equipartition $(P,\alpha)$ with energy $\lambda_k(\Gamma)$, where we choose $\alpha$ based on the values of $\Phi\psi$ across the partition boundary.
\end{proof}

It turns out that requiring that $\p P$ is switching equivalent to $\Gamma$ is more than we need to assume. For $\Gamma \sub E$, let $\mathcal{P}_{\nu}(\Gamma)$ be the set of all $\nu$-partitions $P$ with some subset $\wtil{\Gamma} \sub \p P$ such that $\sigma^{\Gamma}$ is switching equivalent to $\sigma^{\wtil{\Gamma}}$.

\begin{theorem}\label{thm:homology_result}
    Suppose $\Gamma \sub E$. If $\lambda_\nu(\Gamma)$ is the $\nu^{th}$ eigenvalue of $L^{\Gamma}$, then we have 
    \begin{equation}
        \lambda_{\nu}(\Gamma) \leq \inf_{\alpha \in \R^{|\p P|}_{>0}}\Lambda(P,\alpha)
    \end{equation}
     for all $P \in \mathcal{P}_{\nu}(\Gamma)$. Since this is true for any such partition $P$, we obtain
     \begin{equation}
         \lambda_{\nu}(\Gamma) \leq \min_{P \in \mathcal{P}_{\nu}(\Gamma)}\inf_{\alpha \in \R^{|\p P|}_{>0}}\Lambda(P,\alpha).
     \end{equation}
     Furthermore, we have equality $\lambda_{\nu}(\Gamma) = \Lambda(P,\alpha)$ if and only if $P$ is the nodal partition of an eigenvector of $L^{\Gamma}$ with eigenvalue $\lambda_{\nu}(\Gamma)$, in which case $(P,\alpha)$ is a critical equipartition, and the Morse index of $\Lambda$ at $(P,\alpha)$ is the nodal deficiency of $\lambda_{\nu}(\Gamma)$.
\end{theorem}

In the case that $\lambda_{\nu}(\Gamma)$ corresponds to a Courant-sharp eigenvector of $L^{\Gamma}$ with nodal partition $P$, we obtain a critical equipartition $(P,\alpha^c)$ and find that $\Lambda(P,\alpha^c) \leq \Lambda(\wtil{P},\alpha)$ for all $\wtil{P} \in \mathcal{P}_{\nu}(\Gamma)$ and $\alpha>0$. In other words, a Courant-sharp eigenvector of $L^{\Gamma}$ induces a critical equipartition that is minimal in $\mathcal{P}_{\nu}(\Gamma)$.

In the case that $\psi$ is a Courant-sharp eigenvector of $L$ with eigenvalue $\lambda_{\nu}$, we can view $L$ as a signed Laplacian $L^{\Gamma}$ with $\Gamma = \emp$. The set $\mathcal{P}_{\nu}(\emp)$ then consists of all $\nu$-partitions, since any partition $P$ has $\emp \sub \p P$, which induces the all positive signature. Hence, we have shown the following:
\begin{coro}
    If $\psi$ is a non-degenerate Courant-sharp eigenvector of $L$ with eigenvalue $\lambda_{\nu}$, the nodal partition $P$ of $\psi$ is a globally minimal $\nu$-partition, i.e.
    \[\lambda_{\nu} = \Lambda(P,\alpha^c) \leq \Lambda(\wtil{P},\alpha)\]
    for all $\nu$-partitions $\wtil{P}$ and $\alpha \in \R^{|\p P|}_{>0}$, where $(P,\alpha^c)$ is the critical equipartition associated to $\psi$.
\end{coro}

\begin{proof}[Proof of Theorem \ref{thm:homology_result}]
Since $\sigma^{\Gamma}$ is switching equivalent to $\sigma^{\wtil{\Gamma}}$, we know that 
$\innerp{u}{L^{\Gamma}v} = \innerp{\Phi u}{L^{\wtil{\Gamma}}\Phi v}$
for some unitary map $\Phi$.
Let $g(G_k,\alpha)$ be the ground state eigenvector on $G_k$, normalized and chosen such that $g_i(G_k,\alpha)>0$ for all $i \in V_k$. Define 
$\wtil{g}(G_k,\alpha) = \Phi g(G_k,\alpha)$
and choose $c_1,\dots,c_\nu$ such that 
\[\psi = \sum_{k=1}^{\nu}c_k\wtil{g}(G_k,\alpha)\]
is orthogonal to the first $\nu-1$ eigenvectors of $L^{\Gamma}$. We know this can be done because the $\wtil{g}(G_k,\alpha)$ are linearly independent.

By the Rayleigh characterization of eigenvalues, we have 
\begin{equation}\label{rayleigh_quotient}
\lambda_{\nu}(\Gamma) \leq \frac{\innerp{u}{L^{\Gamma}v}}{\innerp{\psi}{\psi}} = \frac{\innerp{\Phi u}{L^{\wtil{\Gamma}}\Phi v}}{c_1^2 + \cdots + c_{\nu}^2}.
\end{equation}
We wish to write $L^{\wtil{\Gamma}}$ in terms of $L^{\p P}(P,\alpha)$ in order to use the fact that $\psi$ is simply a linear combination of orthogonal ground state functions.

First note that 
\begin{equation}\label{pert_by_a}
L^{\wtil{\Gamma}} = L^{\p P} - A,
\end{equation}
where 
\[A_{ij} = \begin{cases}
    2w_{ij}, & (i,j) \in \p P\setminus \wtil{\Gamma}, \\
    0, & \text{else}.
\end{cases}\]
Perturbation of $L^{\p P}$ by $-A$ reverses the signature on all edges in $\p P\setminus \wtil{\Gamma}$, and hence results in $L^{\wtil{\Gamma}}$.
We also have
\begin{equation}\label{pert_by_b}
    L^{\p P}(P,\alpha) = L^{\p P} + B(\alpha)
\end{equation}
where $B = \sum_{(i,j) \in \p P}w_{ij}B_{ij}(\alpha_{ij})$ is the sum of all perturbations in the construction of the parameter dependent partition as in \eqref{parameter_perturbation}.
Combining \eqref{pert_by_a} and \eqref{pert_by_b} gives
$L^{\wtil{\Gamma}} = L^{\p P}(P,\alpha) - A - B(\alpha)$.

We now show that $A + B(\alpha)$ is positive semidefinite. From their definitions, we see that $A + B(\alpha)$ is a symmetric matrix given by
\[(A + B(\alpha))_{ij} = \begin{cases}
    p_i(\alpha), & i=j, \text{ and }i \text{ incident to }\p P \\
    -w_{ij}, & (i,j) \in \wtil{\Gamma} \\
    w_{ij}, & (i,j) \in \p P \setminus \wtil{\Gamma},
\end{cases}\]
where
\[p_i(\alpha) = \sum_{\substack{j>i \\ (i,j) \in \p P}}w_{ij}\alpha_{ij} + \sum_{\substack{j<i \\ (i,j) \in \p P}}\frac{w_{ij}}{\alpha_{ij}}\]
are the diagonal entries consisting of all the potentials. For sake of being explicit we  again taken the convention that when $i<j$ we add $\alpha_{ij}$ to vertex $i$ and $1/\alpha_{ij}$ to vertex $j$. It follows that 
\[\innerp{u}{(A + B(\alpha))u} = \sum_{i \in \p V}p_i(\alpha)u_i^2 + \sum_{(i,j) \in \p P\setminus \wtil{\Gamma}}2w_{ij}u_iu_j - \sum_{(i,j) \in \wtil{\Gamma}}2w_{ij}u_iu_j.\]
Regrouping the terms based on whether an edge lies in $\wtil{\Gamma}$ gives
\begin{align}
    \innerp{u}{(A + B(\alpha))u} = \sum_{(i,j) \in \p P\setminus \wtil{\Gamma}}w_{ij}\alpha_{ij}(u_i + \tfrac{1}{\alpha_{ij}}u_j)^2 + \sum_{(i,j) \in \wtil{\Gamma}}w_{ij}\alpha_{ij}(u_i - \tfrac{1}{\alpha_{ij}}u_j)^2 \geq 0.
\end{align}
Returning to \eqref{rayleigh_quotient}, we compute
\begin{align}
    \lambda_{\nu}(\Gamma) &\leq  \frac{\innerp{\Phi \psi}{L^{\p P}(P,\alpha)\Phi \psi} - \innerp{\Phi \psi}{(A + B(\alpha))\Phi \psi}}{c_1^2 + \cdots + c_{\nu}^2} \\
    &\leq \frac{\innerp{\Phi \psi}{L^{\p P}(P,\alpha)\Phi \psi}}{c_1^2 + \cdots + c_{\nu}^2} \\
    &= \frac{c_1^2\lambda_1(G_1,\alpha) + \cdots + c_{\nu}^2\lambda_1(G_{\nu},\alpha)}{c_1^2 + \cdots + c_{\nu}^2} \\
    &\leq \max_{k=1,\dots,\nu}\lambda_1(G_k,\alpha) \\
    &= \Lambda(P,\alpha)
\end{align}
as desired.

If we have $\lambda_{\nu}(\Gamma) = \Lambda(P,\alpha)$, then Rayleigh's Theorem \cite[Theorem 4.2.2]{hornMatrixAnalysis2012} implies $\psi$ is an eigenvector of $L^{\Gamma}$ with eigenvalue $\lambda_{\nu}(\Gamma)$ and hence that $\Phi \psi$ is an eigenvector of $L^{\wtil{\Gamma}}$ with the same eigenvalue. Furthermore, from the fact that $(A + B(\alpha))\Phi \psi = 0$ we find that 
\[(\Phi\psi)_i + \frac{1}{\alpha_{ij}}(\Phi \psi)_j = 0\]
for all $(i,j) \in \p P \setminus \wtil{\Gamma}$ and 
\[(\Phi\psi)_i - \frac{1}{\alpha_{ij}}(\Phi \psi)_j = 0\]
for all $(i,j) \in \wtil{\Gamma}$. Since each $\alpha_{ij}>0$, this means that $\Phi\psi$ must change sign across each $(i,j) \in \p P\setminus \wtil{\Gamma}$ and not change sign across each $(i,j) \in \wtil{\Gamma}$. Hence, $P$ is the nodal partition of $\Phi \psi$ as an eigenvector of $L^{\wtil{\Gamma}}$ and $P$ is the nodal partition of $\psi$ as an eigenvector of $L^{\Gamma}$. Also, $(P,\alpha)$ is a critical equipartition by Proposition \ref{non-degen_evec_induces_equip}.
\end{proof}

\begin{rem}
    Since $\sigma^{\Gamma}$ is trivially switching equivalent to itself, we obtain the nice result that given some $\nu$-partition $P$,
    \[\lambda_{\nu}(\Gamma) \leq \inf_{\alpha>0}\Lambda(P,\alpha)\]
    for any $\Gamma \sub \p P$. Hence, another lower bound on the energy for a fixed partition $P$ is given by 
    \begin{equation}\max_{\Gamma \sub \p P}\lambda_{\nu}(\Gamma) \leq \inf_{\alpha>0}\Lambda(P,\alpha).
    \end{equation}
\end{rem}

\begin{rem}
    Theorem \ref{thm:homology_result} can be thought of as a discrete analogue of \cite[Theorem 1.6]{berkolaikoHomologySpectralMinimal2024}. This result uses homology to compare different piecewise $C^1$ cuts of a bounded domain. All of our results with switching equivalence could be equivalently phrased in the language of homology in complete parallel with the results in \cite{berkolaikoHomologySpectralMinimal2024}. For a more explicit discussion of this connection we refer the reader to Appendix \ref{sec:switching_equiv_and_homology}.
\end{rem}

\appendix
\section{Connection to the partition Laplacian on surfaces}
\label{app:connection_to_cts}
Defining $L^{\p P}$ as a signed Laplacian is nice because we can use machinery from signed graph theory, but we lose any intuition on how this $L^{\p P}$ is related to the continuum $-\lap^{\p P}$ as mentioned in the introduction. It turns out that we can view $L^{\p P}$ as a discretization of $-\lap^{\p P}$. We sketch the main ideas here.

\subsection{The continuum partition Laplacian}
Let $\Omega \sub \R^2$ be a bounded domain with smooth boundary $\p \Omega$. A $\nu$-partition $P$ of $\Omega$ is a collection of $\nu$ disjoint open connected subsets $\{\Omega_k\}_{k=1}^{\nu}$ of $\Omega$, each with piecewise smooth boundary, such that 
\[\overline{\Omega} = \overline{\bigcup_{k}\Omega_{k}}.\]
We let $\p P$ denote the boundaries of the subdomains that do not coincide with $\p \Omega$, i.e. 
\[\p P = \bigcup_{k=1}^{\nu}\p \Omega_k\setminus \p \Omega.\]
Given an eigenfunction $\psi$ of $-\lap$ with eigenvalue $\lambda_k$ satisfying Dirichlet boundary conditions on $\p \Omega$, we can obtain a bipartite partition $P$ by taking each $\Omega_k$ to be a nodal domain of $\psi$. Roughly speaking, the partition Laplacian is designed to have eigenfunctions with nodal partitions that are not necessarily bipartite, so it is naturally useful when studying non-bipartite minimal partitions.

For a function $u:\Omega \to \R$, we let $u_i = u|_{\Omega_i}$, and we say that $u$ is anticontinuous across $\p P$ if 
\[u_i\big|_{\p \Omega_i \cap \p \Omega_j} = - u_j\big|_{\p \Omega_i \cap \p \Omega_j},\hspace{1cm} \pderiv{u_i}{\nu_i}\big|_{\p \Omega_i \cap \p \Omega_j} = \pderiv{u_j}{\nu_j}\big|_{\p \Omega_i \cap \p \Omega_j}\]
whenever $\p \Omega_i \cap \p \Omega_j \neq \emp$. Then the partition Laplacian $-\lap^{\p P}$ is generated by the bilinear form
\[\innerp{u}{-\lap^{\p P}v} = \sum_{k=1}^{\nu}\int_{\Omega_k}\grad u_k \cdot \grad v_{k} dV\]
and it acts on functions $u$ such that $u_i \in H^1(\Omega_i)$ for each $i$, vanish on $\p \Omega$, and are anticontinuous across $\p P$. If our partition $P$ is bipartite, it follows that $-\lap^{\p P}$ is unitarily equivalent to to $-\lap$. The unitary map uses the bipartite structure and multiplies by $-1$ on certain domains and $1$ on others. See \cite{berkolaikoHomologySpectralMinimal2024} for a more detailed explanation for the partition Laplacian.

\subsection{Anticontinuous Functions on Graphs} Let $G$ be a finite, simple, and connected graph, and let $P = \{G_k\}_{k=1}^{\nu}$ be a $\nu$-partition of $G$. To make sense of an anticontinuous boundary condition along $\p P$, it is helpful to introduce ghost points along the partition boundary.

For each $(i,j) \in \p P$, add two vertices $k_{ij},\ell_{ij}$ along with edges $(i,k_{ij})$ and $(j,\ell_{ij})$. Then remove the original edge $(i,j)$ and set the edge weight on each of the two new edges to be $2w_{ij}$. After following this procedure for each $(i,j) \in \p P$, call the resulting graph $G' = (V',E',w')$.

We say that $u:V' \to \R$ is anticontinuous across $\p P$ if 
\begin{equation}\label{anticontinuity}
    u_{k_{ij}} = -u_{\ell_{ij}}, \hspace{1cm} u_{k_{ij}} - u_i = u_{\ell_{ij}} - u_{j}
\end{equation}
for all $(i,j) \in \p P$.
The first condition is the analogue for $u_i = -u_j$ and the second condition is an analogue for the condition on the normal derivatives. If $L(G')$ is the usual graph Laplacian on $G'$, we define the partition Laplacian $L^{\p P}(G')$ as follows:
\begin{equation}
    (L^{\p P}(G')u)_{i'} =\begin{cases}
            (L(G')u)_{i'}, & i' \in V \\
            u_{k_{ij}} + u_{\ell_{ij}}, & i' = k_{ij} \text{ for some }(i,j) \in \p P\\
            u_{k_{ij}} - u_{i} + u_{j} - u_{\ell_{ij}}, & i' = \ell_{ij} \text{ for some }(i,j) \in \p P.
    \end{cases}
\end{equation}
In other words, we encode the boundary conditions in the rows of $L^{\p P}(G')$ corresponding to the vertices in $V' \setminus V$. The eigenvectors of $L^{\p P}(G')$ that we care about must be anticontinuous, so the entries of $L^{\p P}(G')u$ corresponding to vertices in $V' \setminus V$ will be zero when $u$ is anticontinuous.

Let $B$ be the $|V'| \times |V'|$ diagonal matrix where $B_{ii} = 1$ if $i \in V$ and $B = 0$ otherwise. Then we need to look for the generalized eigenvectors of $L^{\p P}(G')$ with respect to $B$, i.e. we need to find $u$ and $\lambda$ such that 
$L^{\p P}(G')u = \lambda B u$.

This construction of the partition Laplacian works perfectly fine and one can show that given a bipartite partition $P$, the eigenvalues of $L^{\p P}(G')$ with respect to $B$ and the eigenvalues of $L(G)$ are the same, so there is a notion of unitary equivalence between them.

\subsection{Construction of the Partition Laplacian}
It is much nicer if we use a slightly different framework. Given any function $u:V \to \R$ we can extend it uniquely to an anticontinuous function defined on $V'$. We say that $u': V' \to \R$ is the anticontinuous extension of $u$ if 
\[u'_{k_{ij}} = \frac{u_i-u_j}{2}, \hspace{1cm}u'_{\ell_{ij}} = \frac{u_j - u_i}{2}\]
for all $(i,j) \in \p P$, and $u'=u$ everywhere else. This comes from explicitly solving the system \eqref{anticontinuity}. Let $T:\R^{|V|} \to \R^{|V'|}$ be the linear map that takes $u \mapsto u'$. Then if we consider the composition $L^{\p P}(G')T:\R^{|V|} \to \R^{|V'|}$ we find that any rows corresponding to vertices in $V'\setminus V$ must be zero. We can then project $R:\R^{|V'|} \to \R^{|V|}$ by removing the bottom $|V'| - |V|$ rows from $L^{\p P}(G')T$. The result is an operator $RL^{\p P}(G')T:\R^{|V|} \to \R^{|V|}$. 

An explicit but tedious calculation shows that $RL^{\p P}(G')T$ is a symmetric $|V| \times |V|$ matrix that is equal to $L$ at every entry except for those corresponding to $(i,j) \in \p P$, and in those entries the sign is flipped. This is exactly the partition Laplacian as defined in \eqref{def:part_lap}. 

\section{Switching Equivalence and Homology}\label{sec:switching_equiv_and_homology}

The space $C_0(G,\Z_2)$ of 0-chains on $G$ consists of formal linear combinations of vertices of $G$ with and the space $C_1(G,\Z_2)$ of 1-chains consists of formal linear combinations of edges of $G$, both with coefficients taken in $\Z_2$. Since all addition is taken mod 2, note that for any $\Gamma_1,\Gamma_2 \in C_1(G,\Z_2)$ addition is simply the symmetric difference $\Gamma_1+\Gamma_2 = \Gamma_1 \Delta \Gamma_1$.

We define the boundary mapping $\p:C_1(G,\Z_2) \to C_0(G,\Z_2)$ by $\p(i,j) = (i) + (j)$, extended by linearity. The first homology $H_1(G,\Z_2)$ is thus
\[H_1(G,\Z_2) = \ker(\p),\]
since there are no 2-chains on a graph. We note that $\ker(\p)$ is generated by cycles on $G$ so $\dim \ker (\p) = \beta$, the first Betti number of $G$. For $\Gamma \in C_1(G,\Z_2)$ we denote its equivalence class in $H_1(G,\Z_2)$ by $[\Gamma]$. We say that $\Gamma_1$ and $\Gamma_2$ in $C_1(G,\Z_2)$ are \textit{homologous} if and only if $[\Gamma_1+\Gamma_2] = 0 \in H_1(G,\Z_2)$. 

Define a positive semidefinite bilinear form $t$ on $C_1(G,\Z_2)$ by $t(e_i,e_j) = \delta_{ij}$, extended by linearity. We consider the orthogonal complement of $\ker(\p)$ with respect to $t$:
\[\ker(\p)^{\perp} = \{\Gamma \in C_1(G,\Z_2): t(\Gamma,\wtil{\Gamma}) = 0 \text{ for all }\wtil{\Gamma} \in \ker(\p)\}.\]
An explicit computation shows that $\ker(\p)^{\perp}$ is generated by edges that form cuts of $G$. We say that $\Gamma \sub E$ forms a cut of $G$ if there exist two disjoint $S,T \sub V$ with $S \cup T = V$ such that each $e \in \Gamma$ connects one vertex in $S$ to one vertex in $T$. We label the cut space $\mathcal{C}^{\ast}$.

With this estabilished we have the orthogonal decomposition 
\[C_1(G,\Z_2) = H_1(G,\Z_2) \oplus \mathcal{C}^{\ast}.\]
Hence, $\Gamma_1$ and $\Gamma_2$ are homologous if and only if $\Gamma_1 + \Gamma_2 \in \mathcal{C}^{\ast}$. Since every cut forms a bipartite partition of $G$ and addition is equivalent to symmetric difference, we find that $\Gamma_1$ and $\Gamma_2$ are homologous if and only if $\Gamma_1 \Delta \Gamma_2$ is the boundary set of a bipartite partition. This in turn is equivalent to saying that $\sigma^{\Gamma_1}$ is switching equivalent to $\sigma^{\Gamma_2}$ by Lemma \ref{switching_equiv_crit}, so $\sigma^{\Gamma_1}$ and $\sigma^{\Gamma_2}$ are switching equivalent if and only if $\Gamma_1$ and $\Gamma_2$ are homologous.

\printbibliography

\end{document}